\documentclass[notitlepage,reqno,11pt]{amsart}
\usepackage{latexsym,amssymb, epsfig, amsmath,amsfonts, subfigure,amsthm}

\usepackage{rotating}
\usepackage[toc,page]{appendix}
\usepackage{color}
\usepackage{multirow}
\usepackage{relsize}
\usepackage{microtype}

\usepackage[colorlinks, allcolors=blue]{hyperref}

\usepackage[margin=1in]{geometry}

\numberwithin{equation}{section}
 
\newtheorem{lemma}{Lemma}[section]
\newtheorem{theorem}{Theorem}[section]

\newtheorem{prop}{Proposition}[section]
\newtheorem{remark}{Remark}[section]

\newlength{\defbaselineskip}
\newcommand{\setlinespacing}[1]%
           {\setlength{\baselineskip}{#1 \defbaselineskip}}

\newcommand{\R}{{\mathbb R}}
\newcommand{\Z}{{\mathbb Z}}

\def\E{\mathbb{E}}
\def\P{\mathbb{P}}

\newcommand{\sF}{{\mathcal{F}}}
\newcommand{\sI}{{\mathcal{I}}}

\newcommand{\sK}{{\mathcal{K}}}
\newcommand{\sL}{{\mathcal{L}}}

\newcommand{\ep}{\varepsilon}

\newcommand{\tT} {\sigma^N_{_T}}

\newcommand{\bD}{{\mathbf D}}

\newcommand{\bC}{{\mathbf C}}

\newcommand{\bone}{{\mathbf 1}}

\newcommand{\qinq}{\quad\mbox{in}\quad}
\newcommand{\qforallq}{\quad\mbox{for all}\quad}

\newcommand{\non}{\nonumber}

\newcommand{\ttl}{\Large Multi-patch multi-group epidemic model with varying infectivity}

\begin{document}

\title[]{\ttl}
\author[Rapha\"el \ Forien]{Rapha\"el Forien}
\address{INRAE, Centre INRAE PACA, Domaine St-Paul - Site Agroparc
84914 Avignon Cedex
France}
\email{raphael.forien@inrae.fr}

\author[Guodong \ Pang]{Guodong Pang}
\address{Department of Computational Applied Mathematics and Operations Research,
George R. Brown College of Engineering,
Rice University,
Houston, TX 77005}
\email{gdpang@rice.edu}

\author[{\'E}tienne \ Pardoux]{{\'E}tienne Pardoux}
\address{Aix Marseille Univ, CNRS, I2M, Marseille, France}
\email{etienne.pardoux@univ-amu.fr}

\date{\today}

\begin{abstract} 
This paper presents a law of large numbers result, as the size of the population tends to infinity, of SIR stochastic epidemic models,
for a population distributed over $L$ distinct patches (with migrations between them) and $K$  distinct groups (possibly age groups). 
The limit is a set of Volterra-type integral equations, and the result shows the effects of both spatial and population heterogeneity. 
The novelty of the model is that the infectivity of an infected individual is infection age dependent. More precisely, to each  infected individual is attached a random infection-age dependent infectivity function, such that the various random functions attached to distinct individuals are i.i.d.

The proof involves a novel construction of a sequence of i.i.d. processes to invoke the law of large numbers for processes in $\bD$, 
by using the solution of a MacKean-Vlasov type Poisson-driven stochastic equation (as in the propagation of chaos theory). 
We also establish an identity using the Feynman-Kac formula for an adjoint backward ODE. The advantage of this approach is that it assumes much weaker conditions on the random infectivity functions than our earlier work for the homogeneous model  in \cite{FPP2020b}, where standard tightness criteria for convergence of stochastic processes were employed. 
To illustrate this new approach,  we first explain the new proof under the weak assumptions for the homogeneous model, and then describe the multipatch-multigroup model and prove the law of large numbers for that model.  


\end{abstract}


\maketitle



\allowdisplaybreaks

\section{Introduction}
It is well--known that ODE epidemic models are law of large numbers (LLN) limits, as the size of the population 
tends to infinity, of individual based stochastic Markov models, see e.g. \cite{andersson2012stochastic} and Chapter 2 of Part I in \cite{britton2018stochastic}. The Markov property of such stochastic epidemic models requires in particular that the duration of the infectious period of the individuals follows an exponential distribution.
For epidemic models with non--exponential infectious periods, LLNs have been derived  using different methods, see \cite{reinert1995asymptotic} for the SIR model, and  \cite{wang1975limit} for the case of a non--Markov population model. 
Related models of age-structured populations have been studied before.
In \cite{oelschlager_limit_1990}, Oelschl\"ager studies an age-structured birth and death process in which the birth and death rates depend on the age structure of the whole population, and proves a law of large numbers and a central limit theorem for the empirical distribution of the age of individuals in the population.
Similar results have been proved for more general models in \cite{tran_large_2008,meleard_slow_2012} and \cite{hamza_age_2013,fan_convergence_2020}.
Recently, in \cite{PP-2020}, the last two authors have shown the LLN results for various non-Markov models (including SIR, SEIR, SIS, SIRS models) with the infectious periods having any general distribution. The limits in these non-Markov models are systems of Volterra-type integral equations. Of course, without proving LLNs,  Volterra integral equations were already used to describe the epidemic dynamics in the literature, see, for example, \cite{brauer1975nonlinear,cooke1976epidemic,diekmann1977limiting,hethcote1995sis,van2000simple,feng2007epidemiological}.

On the other hand, not only should the infectious period be non-exponentially distributed, but, as was advocated as early as 1927 by Kermack and McKendrick in \cite{KMK}, the infectivity should be allowed to vary with the time elapsed since infection. 
A deterministic integral equation was also developed to describe the age of infection in epidemic models, see  \cite{brauer2008age} and  \cite[Chapter 4.5]{BCF-2019}. 
In \cite{FPP2020b}, the authors have obtained the Kermack--McKendrick model as the large population LLN limit of a stochastic model, where to each infectious individual is attached an independent copy of a random infection age dependent random infectivity function. The last two authors have established a central limit theorem in the same framework, see \cite{PP2020-FCLT-VI}. Furthermore, by tracking the age of infection of all the individuals in that framework, the last two authors in  \cite{PP-2020c} have also established the LLN with the limit being a system of first order partial differential equations (PDEs) with an integral equation as boundary condition, which is consistent with the Kermack and McKendrick PDE model introduced in 1932 \cite{KM32}. 


The present paper contains two novelties. First, the techniques employed in \cite{FPP2020b} for the SIR model with varying infectivity to prove the law of large numbers require that each infectivity function be uniformly bounded and satisfy a regularity condition as imposed in \cite[Assumption 2.1]{FPP2020b}. Specifically, the infectivity function is assumed to have at most finitely many jumps, and to satisfy a uniform continuity assumption between the jumps. 
These conditions were necessary in the proof of convergence of the aggregate infectivity process 
in the space $\bD$ of real-valued c{\`a}dl{\`a}g functions, 
 using standard tightness criteria via the conditions on the modulus of continuity as stated in \cite{billingsley1999convergence} (see the specific criteria used in Lemmas 4.3--4.5 in \cite{FPP2020b}). In fact, the
same technique fails to work for a model with both infection age varying infectivity and recovery age varying susceptibility in \cite{FPPZ}. 
Rather, the only approach with which we were able to prove the result in that case was using a comparison of the model with a sequence of i.i.d. processes, which is easily shown to converge thanks to the law of large numbers for processes in $\bD$, cf. \cite{rao1963law}. The construction of the law of those processes involves the solution of a McKean--Vlasov type Poisson driven stochastic equations
(see equation \eqref{SDE-MKV} for the homogeneous model and equation \eqref{eq:McKV}  for the multi-patch multi-group model), 
as in the ``propagation of chaos'' theory, see Sznitman \cite{sznitman1991topics}. Note that our result can in fact be interpreted as a propagation of chaos result, and that the general approach of the proof is inspired by the work of Chevallier \cite{chevallier_mean-field_2017}. The advantage of this new approach is that the proof requires much weaker assumptions on the infection age dependent infectivity functions than those needed in our previous paper \cite{FPP2020b}: namely, we only assume that those functions belong a.s. to $\bD$ and are uniformly bounded. We first present this new proof of the LLN result from \cite{FPP2020b} for the homogeneous model with varying infectivity in Section  \ref{sec:hom}. 

The main goal (and the second novelty) of the present paper is to adapt this approach to a multipatch-multigroup model. In this model, the population is divided into groups, which are mainly thought of as age groups, and into patches, which are distinct geographical areas. 
Individuals  remain in the same group during the epidemic, while they may move from one patch to another.  
Infectious individuals infect susceptible ones from other groups, but also from other patches. 
Epidemic models with multiple types or groups have been extensively studied, see, e.g., \cite{ball1992final,becker1996immunization,diekmann1990definition,hethcote1985stability,huang1992stability}, as well as models with different geographic areas or patches, see, e.g., \cite{sattenspiel1995structured,allen2007asymptotic,xiao2014transmission,PP-2020b}. 
In \cite{ball2001stochastic}, a stochastic SIR model with multitype individuals that are partitioned into households where infections occur within a household (locally) and between households (globally) is considered. 
A multipatch-multigroup epidemic model was recently studied in \cite{bichara2018multi}, which focuses on the ODEs for the Markovian SEIRS model and its global stability property. Our work is also motivated by the study in Britton et al. \cite{britton2020mathematical} on the influence of population heterogeneity on herd immunity in the recent Covid-19 pandemic. 
Our model with multiple patches and groups captures both spatial and population heterogeneity. 
The groups could partition the population according to various levels of social activity.

In our model, the infection rate is assumed to take a very general form, as given in formula \eqref{eqn-Upsilon}. It allows a different  infectivity rate from each group and patch to others, thus including in particular non-local infections. 
The main reason for allowing non-local infections is that the propagation of an epidemic from one patch to another is partly due to movements of individuals going from home to work and back, as well as those who  visit a given place during holidays or weekends, and then return home. 
These movements cannot be conveniently modeled as migrations, and their effect on the epidemic are infections at distance.

The multi-patch SEIR model with a homogeneous population in each patch but with a constant infectivity rate was recently considered in \cite{PP-2020b}, where both the LLNs and functional central limit theorems (FCLTs) are established. Note that the techniques for the proof of convergence in \cite{PP-2020b} also use the standard tightness criteria as discussed above, and as in \cite{PP-2020} do not require any condition on the exposed and infectious period distributions  (except that the distributions for the initially infected individuals are continuous for the FCLT). In the present paper,  since the infectious periods are induced from the random infectivity functions, for which no regularity conditions are imposed, the distributions of the infectious periods are also completely general. 

The main result for the multipatch-multigroup model is stated in Theorem \ref{thm-LLN}, where the LLN limit is given by a system of Volterra type integral equations.  
Besides the complication of the notation (a double index for the group and the patch), the main difficulty of adapting the new proof to the multipatch-multigroup model is the need for a formula for the proportion (in the large population limit) of susceptibles from group $k$ located in patch $\ell$ at time $t$. In the homogeneous model, this formula is the well-known formula for the solution of a linear one-dimensional ODE (see equation \eqref{eqn-barS-solution} and its use in the proof of Lemma \ref{le:MKV}). In the multipatch-multigroup model, this formula is replaced by  formula \eqref{eq:FKr} in Proposition \ref{pro:FK}. This is the formula for the solution at time $t$ of a forward ODE, which is the law of the location of a susceptible, weighted by an exponential factor taking into account the patches visited between time $0$ and time $t$. That exponential factor is the conditional probability, given the various positions of the individual during the time interval $[0,t]$, of not having been infected by time $t$. The proof of that formula relies upon the Feynman--Kac formula for an adjoint backward ODE, which is established in Lemma \ref{le:FK}. This formula for the proportion (in the large population limit) of susceptibles from group $k$ located in patch $\ell$ plays a crucial role in the subsequent proofs. First of all, it is used to establish the existence of a unique solution to the system of McKean-Vlasov Poisson-driven stochastic equations (equation \eqref{eq:McKV}) in Lemma \ref{le:EU}. Then it is used repeatedly in the proof of Theorem \ref{thm-LLN}, see the proofs of Lemmas \ref{convS}--\ref{convR}. Notably, the proof of the LLNs for the multipatch-multigroup model is much more sophisticated than the same proof for the homogeneous model.

The paper is organized as follows. In Section \ref{sec:hom}, we present the new proof of the LLNs for the homogeneous model, and in Section \ref{sec:multi}, we describe our multipatch-multigroup model, and state the law of large numbers for this model. In Section \ref{sec-proofs}, we provide the proof for that main result. Specifically, in Section \ref{sec-exist-uniq-solution} we prove the existence and uniqueness of a solution to the limiting system of integral equations, in Section \ref{sec-barS-FK}, we derive an expression for the limiting proportion of susceptibles in each patch and group via a Feynman-Kac formula for the associated backward ODEs, in Section \ref{sec-McKV-SDE}, we propose an auxiliary system of Poisson-driven McKean-Vlasov stochastic equations and prove that it has a unique solution, in Section \ref{sec-estimates-iid}, we construct a sequence of i.i.d. processes from the solution to the Poisson-driven stochastic equations and establish estimates for the differences between the original processes and the i.i.d. processes for various quantities, and finally in Section \ref{sec-completing}, we complete the proof of the law of large numbers using the constructed i.i.d. processes and the estimates from previous sections.

\section{The homogeneous model}\label{sec:hom}
We reformulate the SIR epidemic model with varying infectivity,  and obtain the LLN result under weaker assumptions than in the authors' previous paper \cite{FPP2020b}.

Let $\{\lambda_{-j}, j\ge1\}$ and 
$\{\lambda_j,\ j\ge1\}$ be two mutually independent sequences of i.i.d. random elements of $\bD$ (in this paper, $\bD$ denotes the space of c\`adl\`ag paths from $\R_+$ into $\R$, which we equip with the Skorohod $J_1$ topology, see \cite{billingsley1999convergence} for details). $\lambda_{-j}(t)$ is the infectivity at time $t$ of the $j$-th initially infected individual, and $ \lambda_j(t) $ is the infectivity at time $t$ after its time of infection of the $j$-th individual infected after time 0. 
We assume that there exists a deterministic $\lambda^\ast>0$ such that $0\le \lambda_j(t)\le\lambda^\ast$, for all $j\in\Z\backslash\{0\}$ and $t\ge0$, almost surely. We extend $\lambda_j(t)$ for $j\ge1$ to all $t\in\R$, assuming that $\lambda_j(t)=0$ for $t<0$. 
We next define, for each $j\in\Z\backslash\{0\}$,
\begin{align*}
\eta_j:=\sup\{t>0,\ \lambda_j(t)>0\}.
\end{align*}
We denote by $F(t):=\P(\eta_1\le t)$ and $F_0(t):=\P(\eta_{-1}\le t)$ the distribution functions of $\eta_j$ for $j\ge1$ and for $j\le-1$ respectively. Also, let $F_0^c=1-F_0$ and $F^c=1-F$. Finally we define $\bar{\lambda}^0(t)=\E[\lambda_{-1}(t)]$, $\bar{\lambda}(t)=\E[\lambda_1(t)]$.

\subsection{Model and Results}

We split our population in two subsets: those who were infected at time $t=0$, there are $I^N(0)$ of them, and those who were susceptible at time $t=0$, there are $S^N(0)$ of them.
(Hence $R^N(0)=0$, that is, there are no recovered individuals at time 0.)
 We assume that $\bar{I}^N(0):=N^{-1}I^N(0)\to \bar{I}(0)$ and $\bar{S}^N(0):=N^{-1}S^N(0)\to \bar{S}(0)$ a.s., where $(\bar{I}^N(0),\bar{S}(0))\in(0,1)^2$ is deterministic and such that $\bar{I}(0)+\bar{S}(0)\le1$. Note that the random vector 
$(S^N(0),I^N(0))$ is assumed to be independent of the sequence $\{(\lambda_j,Q_j),\ j\ge1\}$ to be defined below.

For $1\le j\le S^N(0)$, we define $A^N_j(t)$ to be the $\{0,1\}$-valued counting process which is zero if the individual $j$ has not been infected by time $t$, and $1$ if he/she has been infected by time $t$. We also define
$\tau^N_j:=\inf\{t>0,\ A^N_j(t)=1\}$. 

The total force of infection in the population at time $t$ is
\begin{align}\label{FNt}
 \mathfrak{F}^N(t)&=\sum_{j=1}^{I^N(0)}\lambda_{-j}(t)+\sum_{j=1}^{S^N(0)}\lambda_j(t-\tau^N_j) 
 \end{align}
Moreover, with the notation $\bar{\mathfrak{F}}^N(t):=N^{-1}\mathfrak{F}^N(t)$, we define the $A^N_j$'s as follows:
\begin{equation}\label{ANt}
 A^N_j(t)=\int_0^t\int_0^\infty{\bf1}_{A^N_j(s^-)=0}{\bf1}_{u\le\bar{\mathfrak{F}}^N(s^-)}Q_j(ds,du)\,,
 \end{equation}
where $\lbrace Q_j,\ j\ge1 \rbrace$ are mutually independent standard Poisson random measures (PRMs) on $\R^2_+$.
Denoting the number of susceptible individuals in the population at time $t$ by $S^N(t)$, we clearly have
\begin{equation}\label{SNt}
 S^N(t)=S^N(0)-\sum_{j=1}^{S^N(0)} A^N_j(t)\,.
 \end{equation}
In addition, the processes $I^N(t)$ and $R^N(t)$ can be written as
\begin{align} 
I^N(t) &= \sum_{j=1}^{I^N(0)} \bone_{\eta_{-j} > t} + \sum_{j=1}^{S^N(0)} \bone_{t \ge \tau^N_j> t- \eta_j} \non \\
& = \sum_{j=1}^{I^N(0)} \bone_{\eta_{-j} > t} + \sum_{j=1}^{S^N(0)} A^N_j(t) 
-  \sum_{j=1}^{S^N(0)} \bone_{ \tau^N_j+ \eta_j \le t}\,,  \label{INt}\\
R^N(t) &= \sum_{j=1}^{I^N(0)} \bone_{\eta_{-j} \le t} +  \sum_{j=1}^{S^N(0)} \bone_{ \tau^N_j+ \eta_j \le t}\,.  \label{RNt}
\end{align}
Note that the system of stochastic equations \eqref{FNt}--\eqref{RNt} uniquely determines the epidemic dynamics. 

We define $\bar{S}^N(t):=N^{-1}S^N(t)$, $\bar{\mathfrak{F}}^N(t):=N^{-1}\mathfrak{F}^N(t)$, 
  $\bar{I}^N(t):=N^{-1}I^N(t)$ and  $\bar{R}^N(t):=N^{-1}R^N(t)$ for $t\ge 0$.  
We prove the following LLN result. Recall that we have assumed that $(\bar{S}^N(0),\bar{I}^N(0))\to(\bar{S}(0),\bar{I}(0))$ a.s. as $N\to\infty$. 

\begin{theorem}\label{LLN-VI} As $N\to\infty$, $(\bar{S}^N,\bar{\mathfrak{F}}^N,\bar{I}^N, \bar{R}^N)\to(\bar{S},\bar{\mathfrak{F}}, \bar{I}, \bar{R})$ in $\bD^4$  in probability, where for $t\ge0$, the limits $(\bar{S},\bar{\mathfrak{F}})$ are the unique solution to the following system of integral equations 
\begin{equation}\label{limitmodel}
\begin{split}
\bar{S}(t)&=\bar{S}(0)-\int_0^t\bar{S}(s)\bar{\mathfrak{F}}(s)ds,\\
\bar{\mathfrak{F}}(t)&=\bar{I}(0)\bar{\lambda}^0(t)+\int_0^t\bar{\lambda}(t-s)\bar{S}(s)\bar{\mathfrak{F}}(s)ds\,,
\end{split}
\end{equation}
and, given the solution $(\bar{S},\bar{\mathfrak{F}})$,  the limits $(\bar{I}, \bar{R})$ are given by the following integral expressions
\begin{equation} \label{limitmodel-IR}
\begin{split}
\bar{I}(t) &= \bar{I}(0) F^c_0(t)  +  \int_0^t F^c(t-s) \bar{S}(s) \bar{\mathfrak{F}}(s)ds\,,\\
\bar{R}(t) &= \bar{I}(0) F_0(t)  +  \int_0^t F(t-s) \bar{S}(s) \bar{\mathfrak{F}}(s)ds\,.
\end{split}
\end{equation}
\end{theorem}

We note that existence and uniqueness of a solution to the system of Volterra equations \eqref{limitmodel} (and hence that of \eqref{limitmodel-IR}) follows from e.g. Theorem~1.2.13  in \cite{brunner_volterra_2017}.

\subsection{The new idea}
The new idea is to associate to a standard PRM $Q$ on $\R_+^2$ the process $A(t)$ which solves
\begin{equation}\label{SDE-MKV}
\begin{split}
A(t)&=\int_0^t\int_0^\infty{\bf1}_{A(s^-)=0}{\bf1}_{u\le\bar{\mathfrak{G}}(s^-)}Q(ds,du),\text{ where}\\
\bar{\mathfrak{G}}(t)&=\bar{I}(0)\bar{\lambda}^0(t)+\bar{S}(0)\E[\lambda(t-\tau)],\text{ with}\\
\tau&=\inf\{t>0,\ A(t)=1\}\,.
\end{split}
\end{equation}
Here $\bar{\lambda}^0(t)=\E[\lambda_{-1}(t)]$, and $\lambda$ is a random element of $\bD$ which is independent of $Q$ (hence also of $\tau$), and has the same law as $\lambda_1$. Hence in particular 
$\E[\lambda(t-\tau)]=\E[\bar{\lambda}(t-\tau)]= \mathbb{E} [ \int_0^t \bar{\lambda}(t-s)dA(s) ]$.

\begin{remark} \label{rem-McKV-hom}
In this system of stochastic equations, one coefficient of the second equation for $\bar{\mathfrak{G}}(t)$ depends on the law of the unknown $\tau$ which is a functional of $A(\cdot)$. 
We call a system of stochastic equations where the law of the unknown function enters the coefficients a McKean--Vlasov stochastic equation. The equation above can be regarded as a Poisson-driven McKean--Vlasov stochastic equation. 

McKean--Vlasov stochastic equations appear naturally in the theory of propagation of chaos, see Sznitman \cite{sznitman1991topics}. In fact, this new approach can be considered as establishing a propagation of chaos result for the times of infection of the initially susceptible individuals.
\end{remark}

Observe that by the first equation in \eqref{limitmodel}, we have the formula
\begin{equation} \label{eqn-barS-solution}
\bar{S}(t)=\bar{S}(0)\exp\left(-\int_0^t \bar{\mathfrak{F}}(s)ds\right)\,, \quad t\ge 0. 
\end{equation}

We first study existence and uniqueness of a solution to the system of equations \eqref{SDE-MKV}. Note that the first component $A$ is random and belongs a.s. to the space of right continuous piecewise constant functions, which are $0$ before the stopping time $\tau$, and then $1$ (with the possibility that $\tau=+\infty$, in which case $A\equiv0$). The second component $\bar{\mathfrak{G}}$ is a deterministic measurable function from $\R_+$ into $[0,\lambda^\ast]$.
\begin{lemma}\label{le:MKV}
Equation \eqref{SDE-MKV} has a unique solution $(A,\bar{\mathfrak{G}})$, which is given by $\bar{\mathfrak{G}}\equiv\bar{\mathfrak{F}}$, where $(\bar{S},\bar{\mathfrak{F}})$ is the unique solution of \eqref{limitmodel}.
\end{lemma}
\begin{proof}
Let $m\in \bD$ be such that $0\le m(t)\le\lambda^\ast$ for all $t\ge0$. We consider the increasing $\{0,1\}$-valued process $A^{(m)}$ defined by
\[ A^{(m)}(t)=\int_0^t  \int_0^\infty {\bf1}_{A^{(m)}(s^-)=0}{\bf1}_{u\le m(s^-)}Q(ds,du)\,,\]
and define $\tau^{(m)}=\inf\{t>0,\ A^{(m)}(t)=1\}$.
Also set
\begin{align*}
	\bar{\mathfrak{G}}^{(m)}(t) &= \bar{I}(0)\bar{\lambda}^0(t)+\bar{S}(0)\E[\lambda(t-\tau^{(m)})].
\end{align*}
We then note that any $ m $ such that $ m = \bar{\mathfrak G}^{(m)} $ yields a solution to \eqref{SDE-MKV}.
Let us then show that $m=\bar{\mathfrak{G}}^{(m)}$ if and only if
the pair $(\bar{S}^{(m)},\bar{\mathfrak{G}}^{(m)})$ solves the system of integral equations \eqref{limitmodel}, where $\bar{S}^{(m)}(t):=\bar{S}(0)e^{-\int_0^tm(r)dr}$ (compare with equation \eqref{eqn-barS-solution}).
The result will then follow from the existence and uniqueness of a solution to \eqref{limitmodel}.
Note that
\begin{equation} \label{eqn-mfG-p1}
\begin{split}
\bar{\mathfrak{G}}^{(m)}(t) &= \bar{I}(0)\bar{\lambda}^0(t)+\bar{S}(0)\E \left[ \int_0^t \bar{\lambda}(t-s)dA^{(m)}(s) \right]\\
 &= \bar{I}(0)\bar{\lambda}^0(t)+\bar{S}(0)\int_0^t \bar{\lambda}(t-s)\P(A^{(m)}(s)=0)m(s)ds\\
 &= \bar{I}(0)\bar{\lambda}^0(t)+\bar{S}(0)\int_0^t \bar{\lambda}(t-s)m(s)e^{-\int_0^sm(r)dr}ds\,.
 \end{split}
 \end{equation}
 The third equation follows from the definition of $A^{(m)}(t)$.
 We thus see that $m=\bar{\mathfrak{G}}^{(m)}$ if and only if the pair $(\bar{S}^{(m)},\bar{\mathfrak{G}}^{(m)})$ solves \eqref{limitmodel}.
 Hence $m=\bar{\mathfrak{G}}^{(m)}$ has a unique solution $m^\ast$, and moreover $m^\ast\equiv\bar{\mathfrak{F}}$, which establishes the Lemma.
\end{proof}

We next define 
$\{(A_j(t),\tau_j),\ j\ge1\}$ as the solution of \eqref{SDE-MKV} with $(Q,\lambda)$ replaced by
$(Q_j,\lambda_j)$. This yields an i.i.d. sequence $\{(A_j(\cdot),\tau_j),\ j\ge1\}$ of random elements of
$\bD\times\R_+$.

We prove the following estimate when using the i.i.d. sequence $\{(A_j(t),\tau_j),\ j\ge1\}$  to approximate $\{(A^N_j(t),\tau_j),\ j\ge1\}$. This has a similar flavor as that established for the model with varying infectivity and susceptibility in
\cite{FPPZ}, while we note the clear distinctions in the model and proof; see also Lemma 6.2 in \cite{FPP-survey}. 
\begin{lemma}\label{le:estim}
There exists a positive constant $C_{T,\lambda^\ast}$ such that, for all $N\ge1$, $0\le t\le T$,
\[ \frac{1}{N}\E\left[\sum_{j=1}^{S^N(0)}\sup_{0\le t\le T}|A^N_j(t)-A_j(t)|\right]\le C_{T,\lambda^\ast}
(\ep_N+2N^{-1/2})\,,\]
where $\ep_N:=\E\left[|\bar{I}^N(0)-\bar{I}(0)|+|\bar{S}^N(0)-\bar{S}(0)|\right]$.
\end{lemma}
\begin{proof}
From \eqref{ANt} and \eqref{SDE-MKV},
\begin{align*}
|A^N_j(t)-A_j(t)| 
&\le\int_0^t\int_{\bar{\mathfrak{F}}^N(s^-)\wedge\bar{\mathfrak{F}}(s^-)}^{\bar{\mathfrak{F}}^N(s^-)\vee\bar{\mathfrak{F}}(s^-)} Q_j(ds,du). 
\end{align*}
Since the right hand side is non-decreasing,
\begin{align*}
	\sup_{0\le r\le t} |A^N_j(r)-A_j(r)|&\le
	\int_0^t\int_{\bar{\mathfrak{F}}^N(s^-)\wedge\bar{\mathfrak{F}}(s^-)}^{\bar{\mathfrak{F}}^N(s^-)\vee\bar{\mathfrak{F}}(s^-)} Q_j(ds,du).\\
\end{align*}
Taking expectations on both sides then yields
\begin{align}\label{estim1}
	\E\left[\sup_{0\le r\le t} |A^N_j(r)-A_j(r)|\right]&\le\E \left[ \int_0^t|\bar{\mathfrak{F}}^N(s)-\bar{\mathfrak{F}}(s)|ds \right]\,.
\end{align}
Next, using \eqref{FNt}, \eqref{limitmodel} and \eqref{SDE-MKV},
\begin{multline}
\E\left[|\bar{\mathfrak{F}}^N(t)-\bar{\mathfrak{F}}(t)|\right]\le
\E\left[\frac{1}{N}\left|\sum_{j=1}^{I^N(0)}(\lambda_{-j}(t)-\bar{\lambda}^0(t))\right|+\frac{1}{N}\left|\sum_{j=1}^{S^N(0)}(\lambda_j(t-\tau^N_j)-\E[\lambda_j(t-\tau_j)])\right|\right]\\
+\lambda^\ast\E\left[|\bar{I}^N(0)-\bar{I}(0)|+|\bar{S}^N(0)-\bar{S}(0)|\right]\,. \label{estim2}
\end{multline}
Since the $\lambda_{-j}$'s are mutually independent, and globally independent of $I^N(0)$,
\begin{align}
\E\left[\frac{1}{N}\left|\sum_{j=1}^{I^N(0)}(\lambda_{-j}(t)-\bar{\lambda}^0(t))\right|\right]
&\le \frac{1}{N}\E\left[\left|\sum_{j=1}^{I^N(0)}(\lambda_{-j}(t)-\bar{\lambda}^0(t))\right|^2\right]^{1/2}\non\\
&\le \frac{\lambda^\ast}{\sqrt{N}}\,.\label{estim3}
\end{align}
Moreover, 
\begin{multline}
\E\left[\frac{1}{N}\left|\sum_{j=1}^{S^N(0)}(\lambda_j(t-\tau^N_j)-\E[\lambda_j(t-\tau_j)])\right|\right]\\
\le
\E\left[\frac{1}{N}\left|\sum_{j=1}^{S^N(0)}(\lambda_j(t-\tau^N_j)-\lambda_j(t-\tau_j))\right|\right]
+\E\left[\frac{1}{N}\left|\sum_{j=1}^{S^N(0)}(\lambda_j(t-\tau_j)-\E[\lambda_j(t-\tau_j)])\right|\right]
\label{estim41}
\end{multline}
Since the sequence $ \lbrace (\lambda_j(\cdot), \tau_j), j \geq 1\rbrace $ is i.i.d., independent of $S^N(0)$ and $ \lambda_j \leq \lambda^* $ almost surely,
\begin{align*}
	\E\left[\frac{1}{N}\left|\sum_{j=1}^{S^N(0)}(\lambda_j(t-\tau_j)-\E[\lambda_j(t-\tau_j)])\right|\right] \leq \frac{\lambda^\ast}{\sqrt{N}}.
\end{align*}
Note also that
\begin{equation*}
\E\left[\frac{1}{N}\left|\sum_{j=1}^{S^N(0)}[\lambda_j(t-\tau^N_j)-\lambda_j(t-\tau_j)\right|\right]
\le \E\left[ |\lambda_j(t-\tau^N_j)-\lambda_j(t-\tau_j)| \right]\,.
\end{equation*}
On the other hand, using Markov's inequality and the fact that $ A_j $ and $ A^N_j $ are $\{0,1\}$-valued,
\begin{align*}
	\E \left[ \left| \lambda_j(t-\tau^N_j)-\lambda_j(t-\tau_j) \right| \right] &\leq \lambda^* \P(\tau^N_j\wedge t\not=\tau_j\wedge t) \\
	&= \lambda^* \P\left(\sup_{0\le r\le t}|A^N_j(r)-A_j(r)| \geq 1 \right) \\
	&\leq \lambda^* \E\left[\sup_{0\le r\le t}|A^N_j(r)-A_j(r)|\right].
\end{align*}
Combining  \eqref{estim41} with the last three inequalities, we obtain that
\begin{align} \label{estim4}
	\E\left[\frac{1}{N}\left|\sum_{j=1}^{S^N(0)}(\lambda_j(t-\tau^N_j)-\E[\lambda_j(t-\tau_j)])\right|\right] \leq \lambda^* \E\left[\sup_{0\le r\le t}|A^N_j(r)-A_j(r)|\right] + \frac{\lambda^*}{\sqrt{N}}.
\end{align}
 It now follows from 
\eqref{estim1}, \eqref{estim2}, \eqref{estim3} and \eqref{estim4} that
\[ \E\left[\sup_{0\le r\le t} |A^N_j(r)-A_j(r)|\right]\le\lambda^\ast(\ep_N+2N^{-1/2}) t+
\lambda^\ast\int_0^t\E\left[\sup_{0\le r\le s}|A^N_j(r)-A_j(r)|\right]ds\,.\]
The result, with $C_{T,\lambda^\ast}:=\lambda^\ast T\exp(\lambda^\ast T)$ now follows from Gronwall's Lemma.
\end{proof}

Note that since $\sup_{0\le t\le T}\big|A^N_j(t)-A_j(t)\big|$ is either zero or else $1$, this Lemma implies that
\[ \P\left(\sup_{0\le t\le T}\big|A^N_j(t)-A_j(t)\big|\not=0\right)\le C(\ep_N+2N^{-1/2})\,,\]
and also
\[ \P\left(\tau^N_j\wedge T\not=\tau_j\wedge T\right)\le C(\ep_N+2N^{-1/2})\,.\]

\begin{proof}[Completing the proof of Theorem \ref{LLN-VI}]

Now let us turn back to \eqref{FNt} and write
\begin{align*}
	\bar{\mathfrak F}^N(t) &= \frac{1}{N}\sum_{j=1}^{I^N(0)}\lambda_{-j}(t) + \frac{1}{N}\sum_{j=1}^{S^N(0)}\lambda_j(t-\tau_j) + \frac{1}{N}\sum_{j=1}^{S^N(0)}\left(\lambda_j(t-\tau^N_j)-\lambda_j(t-\tau_j)\right), \\
	&=: \bar{\mathfrak{F}}^N_0(t) + \bar{\mathfrak{F}}^N_1(t) + V^N(t).
\end{align*}
It follows from \cite{rao1963law} that, as $N\to\infty$, the first two terms converge a.s. in $\bD$, 
\begin{align*}
	(\bar{\mathfrak{F}}^N_i(t), t \in [0,T]) \to (\bar{\mathfrak{F}}_i(t), t \in [0,T]), \quad i \in \lbrace 0, 1 \rbrace,
\end{align*}
where $ \bar{\mathfrak{F}}_0(t)=\bar{I}(0)\bar{\lambda}^0(t) $ and $ \bar{\mathfrak{F}}_1(t) = \bar{S}(0)\E[\lambda_1(t-\tau)] $.
It remains to consider the error term $ V^N $, which tends to $0$ locally uniformly in $t$ in probability, thanks to Lemma \ref{le:estim}. Indeed
\begin{align*}
\E \left[ \left|\frac{1}{N}\sum_{j=1}^{S^N(0)}\left(\lambda_j(t-\tau^N_j)-\lambda_j(t-\tau_j)\right)\right|\right]
&\le \E \left[ \frac{\lambda^\ast}{N}\sum_{j=1}^{S^N(0)}\P\left(\tau^N_j\wedge t\not=\tau_j\wedge t\right) \right]\\
&\le\lambda^\ast C(\ep_N+2N^{-1/2}),
\end{align*}
which tends to zero as $N\to\infty$. 
Thus we have shown that  $(\bar{S}^N,\bar{\mathfrak{F}}^N_0,\bar{\mathfrak{F}}^N_1,V^N)\to(\bar{S},\bar{\mathfrak{F}}_0,\bar{\mathfrak{F}}_1,0)$ in $\bD^3$ in probability as $N \to \infty$. It follows from Lemma \ref{le:continuity} below that $\bar{\mathfrak{F}}_1$ is continuous, hence as $N\to\infty$,
\[\bar{\mathfrak{F}}^N=\bar{\mathfrak{F}}^N_0+\bar{\mathfrak{F}}^N_1 + V^N \to\bar{\mathfrak{F}}=\bar{\mathfrak{F}}_0+\bar{\mathfrak{F}}_1\]
in $\bD$ in probability.

Let us now show the convergence of $ (\bar{I}^N, \bar{R}^N) $ to $ (\bar{I}, \bar{R}) $, defined in \eqref{limitmodel-IR}.
By the law of large numbers,
\begin{equation}\label{convI0}
N^{-1}\sum_{j=1}^{I^N(0)} \bone_{\eta_{-j} > t} \to \bar{I}(0) F^c_0(t) \qinq \bD
\end{equation}
in probability as $N\to \infty$. 
By the above proof, 
observing that 
\[
\E[A_j(t)] = 1- \exp\left( -\int_0^t \bar{\mathfrak{F}}(s)ds \right)
\]
and using the expressions of $\bar{S}(t)$ in \eqref{limitmodel} and \eqref{eqn-barS-solution}, 
we also have 
\[ N^{-1}\sum_{j=1}^{S^N(0)} A^N_j(t) \to \int_0^t\bar{S}(s)\bar{\mathfrak{F}}(s)ds \qinq \bD\]
in probability as $N\to \infty$. 
By an argument similar to the derivation of \eqref{eqn-mfG-p1},
 we obtain
\begin{align} \label{eqn-convI-keystep}
\bar{S}(0) \E[ \bone_{ \tau+ \eta_1 \le t}] & =  \bar{S}(0) \E\left[\int_0^t F(t-s)d A(s)\right] \non\\ 
& =  \bar{S}(0)\int_0^t F(t-s) \bar{\mathfrak{F}}(s) e^{-\int_0^s \bar{\mathfrak{F}}(r)dr} ds  \non \\
&= \int_0^t F(t-s) \bar{S}(s) \bar{\mathfrak{F}}(s)ds \,. 
\end{align}
Then by  Lemma \ref{le:estim} and LLN of i.i.d. random elements in $\bD$,  we obtain 
\begin{equation}\label{convI1}
N^{-1}  \sum_{j=1}^{S^N(0)} \bone_{ \tau^N_j+ \eta_j \le t} \to  \int_0^t F(t-s) \bar{S}(s) \bar{\mathfrak{F}}(s)ds \qinq \bD
\end{equation}
in probability as $N\to\infty$. 
Moreover again by Lemma \ref{le:continuity}, the limit is continuous, hence adding
\eqref{convI0} and \eqref{convI1}, we conclude that $\bar{I}^N\to\bar{I}$ in $\bD$ in probability. Since 
$\bar{R}^N(t)=1-\bar{S}^N(t)-\bar{I}^N$, $\bar{R}(t)=1-\bar{S}(t)-\bar{I}(t)$ and $\bar{S}$ is continuous, we conclude
that  $(\bar{S}^N,\bar{\mathfrak{F}}^N,\bar{I}^N,\bar{R}^N)\to(\bar{S},\bar{\mathfrak{F}},\bar{I},\bar{R})$  $\bD^4$ in probability as $N \to \infty$. We have used above several times the fact that, while $(f_n,g_n)\to (f,g)$ in $\bD^2$ does not imply that $f_n+g_n\to f+g$, this last statement holds if in addition either $f$ or $g$ is continuous. See the comment following Lemma \ref{le:continuity} below.
This completes the proof.
\end{proof}

\section{The multipatch-multigroup model}\label{sec:multi}

We assume that the population of size $N$ is split into $K$ groups, and distributed on $L$ distinct patches. 
Let $S^{N,\ell}_k(t)$, $I^{N,\ell}_k(t)$ and $R^{N,\ell}_k(t)$ denote the numbers of susceptible, infected and recovered individuals from group $k$ who are in patch $\ell$ at time $t$, respectively. Let $B^{N,\ell}_k(t)$ be the number of individuals in group $k$ and in patch $\ell$ at time $t$, i.e.,
\[  B^{N,\ell}_k(t)=S^{N,\ell}_k(t)+I^{N,\ell}_k(t)+R^{N,\ell}_k(t)\,.\]
Note that in our model, the total population size, $\sum_{\ell=1}^L\sum_{k=1}^K B^{N,\ell}_k(t)$ is fixed, and equal to $ N $.

We are given a collection of positive numbers $\{\bar{S}^\ell_k(0), \bar{I}^\ell_k(0), \bar{R}^\ell_k(0),\ 1\le \ell\le L,\ 1\le k\le K\}$ which are such that, with $\bar{S}_k(0)=\sum_{\ell=1}^L\bar{S}^\ell_k(0)$, $\bar{I}_k(0)=\sum_{\ell=1}^L\bar{I}^\ell_k(0)$ and $\bar{R}_k(0)=\sum_{\ell=1}^L\bar{R}^\ell_k(0)$, $\bar{S}_k(0)+\bar{I}_k(0)+\bar{R}_k(0)=1$. For each $1\le\ell\le L$ and $1\le k\le K$, we let $\bar{B}^\ell_k(0):=\bar{S}^\ell_k(0)+\bar{I}^\ell_k(0)+ \bar{R}^\ell_k(0)$, and we assume that $\min_{\ell,k}\bar{B}^\ell_k(0)>0$.

We choose arbitrarily the $3K$ integers $S^N_k(0), I^N_k(0), R^N_k(0)$ for $1\le k\le K$ in such a way that
$\sum_k[S^N_k(0)+I^N_k(0)+R^N_k(0)]=N$ and for each $1\le k\le K$, we have
\[ |S^N_k(0)-N\bar{S}_k(0)|\vee |I^N_k(0)-N\bar{I}_k(0)|\vee |R^N_k(0)-N\bar{R}_k(0)|\le 1\,.\]

For each $1\le k\le K$, the susceptible (resp. infected, resp. recovered) individuals from group $k$ jump from patch to patch according to a continuous time jump Markov process, that we shall denote by $X_k(t)$ (resp. $Y_k(t)$, resp. $Z_k(t)$), and whose dynamics will be specified below. We now specify the initial conditions of those processes, together with the initial populations in each compartment, patch and group.
The three collections of r.v.'s $\{X_{j,k}(0),\ 1\le j\le S^N_k(0)\}$, $\{Y_{j,k}(0),\ 1\le j\le I^N_k(0)\}$ and $\{Z_{j,k}(0),\ 1\le j\le R^N_k(0)\}$ are mutually independent, each one is i.i.d., and the distributions are specified as follows:
\[ \P(X_{j,k}(0)=\ell)=\frac{\bar{S}^\ell_k(0)}{\bar{S}_k(0)},\quad \P(Y_{j,k}(0)=\ell)=\frac{\bar{I}^\ell_k(0)}{\bar{I}_k(0)},\quad \P(Z_{j,k}(0)=\ell)=\frac{\bar{R}^\ell_k(0)}{\bar{R}_k(0)}\,.\] Finally, for $1\le\ell\le L$ and $1\le k\le K$, $S^{N,\ell}_k(0)$ (resp. $I^{N,\ell}_k(0)$, resp. $R^{N,\ell}_k(0)$) is the cardinal of the set $\{j: X_{j,k}(0)=\ell\}$ (resp. $\{j: Y_{j,k}(0)=\ell\}$, resp. $\{j: Z_{j,k}(0)=\ell\}$).
In fact we shall not use the processes $Z_{j,k}$ below. However, it is convenient to define the repartition of the initially recovered individuals in the various patches as we did for the initially susceptible and for the initially infected ones.

It clearly follows from the above definitions and the law of large numbers that as $N\to \infty$, $\bar{S}^{N,\ell}_k(0) = N^{-1}S^{N,\ell}_k(0) \to \bar{S}^{\ell}_k(0)$, $ \bar{I}^{N,\ell}_k(0) = N^{-1}I^{N,\ell}_k(0) \to \bar{I}^{\ell}_k(0)$ and $ \bar{R}^{N,\ell}_k(0) = N^{-1}R^{N,\ell}_k(0) \to \bar{R}^{\ell}_k(0)$ a.s.

Moreover, we assume that the above defined initial random variables $S^{N,\ell}_k(0)$, $I^{N,\ell}_k(0)$ and $R^{N,\ell}_k(0)$ 
are independent of the random objects $\lambda_{j,k}$, $X_k, Y_k$ and $Q_{j,k}^{\ell}$ defined below that dictate the dynamics after time 0, as well as of the processes $X_{j,k}(t), Y_{j,k}(t), Z_{j,k}(t)$.

While susceptible, an individual from group $k$ moves from patch to patch according to a time-inhomogeneous Markov process $X_k(t)$, with jump rates $\nu_{S,k}^{\ell, \ell'}(t)$ and transition function $p_k^{\ell,\ell'}(s,t)=\P(X_k(t)=\ell'|X_k(s)=\ell)$, and  while infectious, an individual from group $k$ moves from patch to patch according to a time-inhomogeneous Markov process $Y_k(t)$ with jump rates $\nu_{I,k}^{\ell, \ell'}(t)$ and transition function
\begin{align} \label{transition_function_q}
	q_k^{\ell,\ell'}(s,t)=\P(Y_k(t)=\ell'|Y_k(s)=\ell).
\end{align}
Similarly, the recovered individuals migrate with rates $\nu_{R,k}^{\ell, \ell'}(t)$.
We assume that those movements of the various individuals are mutually independent. The time inhomogeneity may be due to restrictions of movements  imposed by the authorities during the epidemic. We assume that all the rates are locally bounded, 
  i.e., for any $T>0$,
 \[\sup_{0\le t\le T,\, k, \ell, \ell'}[\nu_{S,k}^{\ell, \ell'}(t)+\nu_{I,k}^{\ell, \ell'}(t)+\nu_{R,k}^{\ell, \ell'}(t)]<\infty\,.\] 
 We shall write $X^{s,\ell}_{j,k}(t)$ (resp. $Y^{s,\ell}_{j,k}(t)$) for the position at time $t$ of the individual $j$ from group $k$ if it is susceptible (resp. infected) during the time interval $(s,t)$, and was in patch $\ell$ at time $s$. 
 $X_{j,k}(t)$ (resp. $Y_{j,k}(t)$) will denote the position of the individual $j$ at time $t$, in case that individual is initially susceptible (resp. infected) and is still susceptible (resp. infected) at time $t$. 

The initially infected individual $j$ from group $k$ has at time $t\ge0$ the infectivity $\lambda_{j, k}(t)$ (recall that in this case $j\le -1$), while an initially susceptible individual $j$ from group $k$ who is infected at time 
$\tau_{j,k}^N$ has at time $t$ the infectivity 
$\lambda_{j,k}(t-\tau_{j,k}^N)$. The random functions $\{\lambda_{j,k},\ j\in\Z,\ 1\le k\le K\}$ are mutually independent.
For each $1\le k\le K$,  $\{\lambda_{j,k},\ j\le -1\}$ have the same law, as well as  $\{\lambda_{j,k},\ j\ge 1\}$. But the laws of $\lambda_{-1,k}$ and $\lambda_{1,k}$ are different:
the infectivity of the initially infected individuals has a different distribution from that of the newly infected individuals.
The infectivity depends upon the group $k$, which is quite natural in the case of age groups, since the reaction to infection depends upon the age. On the contrary, the infectivity does not depend upon the patch where the infected individual finds itself. We assume that for any $j\in\Z \backslash \{0\}$ and $1\le k\le K$, $\lambda_{j, k}$ has trajectories in $\bD$, and moreover, $0\le \lambda_{j, k}(t)\le\lambda^\ast$ a.s. for all $j\in\Z \backslash\{0\}$, $1\le k\le K$ and $t\ge0$, where $\lambda^\ast>0$ is a fixed constant.
\begin{remark}
It would be more satisfactory to assume that $\lambda^\ast$ is a r.v. with some finite moment. Unfortunately,
at this stage we are not able to extend our proofs to such a situation.
\end{remark}

We also define, for each $j\in\Z\backslash\{0\}$ and $k=1,\ldots,K$, $\eta_{j,k}=\sup\{t>0,\ \lambda_{j,k}(t)>0\}$. 
 We assume that $\lambda_{j,k}(t)=0$  for $t<0$. 
The random variable $\eta_{j,k}$ represents  the infected periods of newly and initially infected individuals for $j>0$ and $j<0$, respectively.  
Note that the infected period may include both exposed and infectious periods. 

We denote by $F_k(t):=\P(\eta_{1,k}\le t)$ and $F^0_k(t)=\P(\eta_{-1,k}\le t)$ the distribution function of $\eta_{j,k}$ for $j\ge1$ and for $j\le-1$, respectively. 

 Under the  i.i.d. assumption of the random functions $\{\lambda_{j,k}(\cdot)\}_{j\ge 1}$, the sequence of variables $\{\eta_{j,k}\}_{j\ge 1}$ is i.i.d. for each type $k$. 
Similarly, the sequence  $\{\eta_{j,k}\}_{j\le -1}$ is also i.i.d. for each type $k$, and independent of $\{\lambda_{j,k}(\cdot)\}_{j\ge 1}$.  We moreover define $\bar{\lambda}_k(t):=\E[\lambda_{1,k}(t)]$ and $\bar{\lambda}^0_k(t):=\E[\lambda_{-1,k}(t)]$. 


The total force of infection delivered at time $t$ by the individuals of group $k$ in patch $\ell$ is given by
\begin{equation}\label{eq:FN} 
\mathfrak{F}^{N,\ell}_{k}(t)=\sum_{j=1}^{I^{N}_k(0)}\lambda_{-j,k}(t){\bf1}_{Y_{j,k}(t)=\ell}+
\sum_{j=1}^{S^{N}_k(0)}\lambda_{j,k}(t-\tau_{j,k}^N)\sum_{\ell'} {\bf1}_{X_{j,k}(\tau_{j,k}^N)=\ell'}{\bf1}_{Y^{\tau_{j,k}^N,\ell'}_{j,k}(t) =\ell}\,.
\end{equation}
Note that since $\lambda_{j,k}(t-\tau^N_{j,k})$ is zero for $t<\tau^N_{j,k}$, we do not care about the fact that the last factor above is defined only for $t\ge\tau^N_{j,k}$.
What depends upon the patch is the contact rate. 
 We assume that a susceptible of patch $\ell$ and group $k$ has contacts with possibly infectious individuals of patch $\ell'$ and group $k'$ at rate 
$\beta^{\ell,\ell'}_{k,k'}(t)$  at time $t$, and that there exists a constant $\beta^*>0$ such that $\beta^{\ell,\ell'}_{k,k'}(t) \le \beta^*$ for all $\ell, \ell', k, k'$ and $t\ge 0$.    The functions $\beta^{\ell,\ell'}_{k,k'}(t)$ dictate the dependence upon the pairs $(k,k')$ and $(\ell,\ell')$ in the contact rates. This include the so-called ``infection at distance", where $\beta^{\ell,\ell'}_{k,k'}(t)$ can take very small values. As we have discussed in the introduction, this is an approach to model movements  between residence and working place, or during a weekend or holiday, which would be difficult to model as migrations between different patches. 
We assume also that for a given parameter $\gamma\in[0,1]$, the rate at which new infections affect the individuals
from group $k$ in patch $\ell$ at time $t$ is
\begin{align}  \label{eqn-Upsilon}
\Upsilon^{N,\ell}_k(t)&=\frac{S^{N,\ell}_k(t)}{N^{1-\gamma}(B^{N,\ell}_k(t))^\gamma}
\sum_{k'}\sum_{\ell'}\beta^{\ell,\ell'}_{k,k'}(t)\mathfrak{F}^{N,\ell'}_{k'}(t)\\
&=\Big( \frac{B^{N,\ell}_k(t)}{N}\Big)^{1-\gamma}\frac{S^{N,\ell}_k(t)}{B^{N,\ell}_k(t)}\sum_{k'}\sum_{\ell'}\beta^{\ell,\ell'}_{k,k'}(t)\mathfrak{F}^{N,\ell'}_{k'}(t) \,.\non
\end{align}
Let us explain the role of the parameter $\gamma$. In case $\gamma=1$, infected individuals from patch $\ell'$ and group $k'$ meet individuals from patch $\ell$ and group $k$ at rate $\beta^{\ell,\ell'}_{k,k'}$, and that encounter results in a new infection if the partner in the encounter is susceptible, which happens with probability $\frac{S^{N,\ell}_k(t)}{B^{N,\ell}_k(t)}$. In that case $\gamma=1$, the fact that the rate of encounters in patch $\ell$ of individuals from group $k$ is likely to depend upon the abondance of such individuals (one is likely to meet individuals at a higher rate in a densely populated area than in a desert) can be adjusted to the expected density of such individuals by the parameter $\beta^{\ell,\ell'}_{k,k'}$, while in cases $\gamma<1$ that rate of encounter is modulated by the actual (random) local density of individuals of group $k$. The effect of this modulation is highest for $\gamma=0$.
By convention, whenever $B^{N,\ell}_k(t)=0$, we set $ \Upsilon^{N,\ell}_{k}(t)=0$. 


Also, let 
\[
\bar\Gamma^{N,\ell}_k(t) := \frac{1}{N^{1-\gamma}(B^{N,\ell}_k(t))^\gamma}
\sum_{k'}\sum_{\ell'}\beta^{\ell,\ell'}_{k,k'}(t)\mathfrak{F}^{N,\ell'}_{k'}(t)\,,
\]
which is the force of infectivity to each susceptible of group $k$ in patch $\ell$. 
We define 
$\bar{B}^{N,\ell}_k(t):=N^{-1}B^{N,\ell}_k(t)$, $\bar{\mathfrak{F}}^{N,\ell}_{k}(t):=N^{-1}\mathfrak{F}^{N,\ell}_{k}(t)$. 
Then we have 
\begin{equation*} \label{eqn-bar-Gamma-n-lk}
\bar\Gamma^{N,\ell}_k(t)  = \frac{1}{(\bar{B}^{N,\ell}_k(t))^\gamma}
\sum_{k'}\sum_{\ell'}\beta^{\ell,\ell'}_{k,k'}(t) \bar{\mathfrak{F}}^{N,\ell'}_{k'}(t)\,.
\end{equation*}

Now we model the infection of each initially susceptible individual. For each $1\le k\le K$, $1\le j\le S^N_k(0)$ and $ 1 \leq \ell \leq L $, let
\begin{align} \label{eqn-A-N-jk}
A^{N,\ell}_{j,k}(t)=\int_0^t\int_0^\infty{\bf1}_{A^{N}_{j,k}(s^-)=0}{\bf1}_{X_{j,k}(s)=\ell}{\bf1}_{u\le \bar\Gamma^{N,\ell}_{k}(s^-)} Q^\ell_{j,k}(ds,du)\,,
\end{align}
where $\{Q^\ell_{j,k},\  k \in \sK, \ell \in \sL, j\ge1\}$ are mutually independent standard PRMs on $\R^2_+$ 
and $A^N_{j,k}(t)=\sum_\ell A^{N,\ell}_{j,k}(t)$. 
$A^N_{j,k}(t)=1$ if and only if the individual $j$ from group $k$ has been infected on the time interval $(0,t]$. Otherwise,
$A^N_{j,k}(t)=0$. Recall that $\tau^N_{j,k}$ denotes the time at which the initially susceptible individual $j$ from group $k$ is infected. We have $\tau^N_{j,k}=\inf\{t>0, A^N_{j,k}(t)=1\}$.
If $A^N_{j,k}(t)=1$, the unique $\ell$ such that $A^{N,\ell}_{j,k}(t)=1$ is the patch where the individual $j$ 
from group $k$ has been infected.

Besides \eqref{eq:FN}, the evolution of the epidemic is characterized by the dynamics of 
$S^{N,\ell}_k(t)$, $I^{N,\ell}_k(t)$ and $R^{N,\ell}_k(t)$: 
\begin{equation}\label{eq:SN}
\begin{split}
S^{N,\ell}_k(t)&=S^{N,\ell}_k(0)-\sum_{j=1}^{S^N_k(0)}A^{N,\ell}_{j,k}(t)
-\sum_{\substack{\ell'=1 \\ \ell'\neq \ell}}^L P^{\ell,\ell'}_{S,k}\left(\int_0^t\nu^{\ell,\ell'}_{S,k}(s)S^{N,\ell}_k(s)ds\right)\\&\quad
+\sum_{\substack{\ell'=1 \\ \ell'\neq \ell}}^L P^{\ell',\ell}_{S,k}\left(\int_0^t\nu^{\ell',\ell}_{S,k}(s)S^{N,\ell'}_k(s)ds\right),
\end{split}
\end{equation}

\begin{equation}\label{eq:IN}
\begin{split}
I^{N,\ell}_k(t) &= I^{N,\ell}_k(0)   +\sum_{j=1}^{S^N_k(0)}A^{N,\ell}_{j,k}(t)  -  \sum_{\ell'=1}^L \sum_{j=1}^{I^{N,\ell'}_k(0)} \bone_{\eta_{-j,k} \le t} \bone_{Y^{0,\ell'}_{-j,k}(\eta_{-j,k}) =\ell}   \\
& \quad -  \sum_{j=1}^{S^{N}_{k}(0)} \bone_{\tau^{N}_{j, k}+ \eta_{j,k} \le t} \sum_{\ell'}\bone_{X_{j,k}(\tau^N_{j,k}) =\ell'}
 \bone_{Y^{\tau^{N}_{j, k},\ell'}_{j,k}\!\!(\tau^{N}_{j, k}+\eta_{j,k})=\ell}      \\
& \quad 
 -\sum_{\substack{\ell'=1 \\ \ell'\neq \ell}}^L P^{\ell,\ell'}_{I,k}\left(\int_0^t\nu^{\ell,\ell'}_{I,k}(s)I^{N,\ell}_k(s)ds\right) 
+\sum_{\substack{\ell'=1 \\ \ell'\neq \ell}}^L P^{\ell',\ell}_{I,k}\left(\int_0^t\nu^{\ell',\ell}_{I,k}(s)I^{N,\ell'}_k(s)ds\right),
\end{split}
\end{equation}
\begin{equation}\label{eq:RN}
\begin{split}
R^{N,\ell}_k(t) &= R^{N,\ell}_k(0)+\sum_{\ell'=1}^L \sum_{i=1}^{I^{N,\ell'}_k(0)} \bone_{\eta_{-j,k} \le t} \bone_{Y^{0,\ell'}_{-j,k}(\eta_{-j,k}) =\ell}   \\
& \quad +  \sum_{j=1}^{S^{N}_{k}(0)} \bone_{\tau^{N}_{j, k}+ \eta_{j,k} \le t} \sum_{\ell'}\bone_{X_{j,k}(\tau^N_{j,k}) =\ell'}
 \bone_{Y^{\tau^{N}_{j, k},\ell'}_{j,k}\!\!(\tau^{N}_{j, k}+\eta_{j,k})=\ell}     \\
& \quad - \sum_{\substack{\ell'=1 \\ \ell'\neq \ell}}^L P^{\ell,\ell'}_{R,k} \left( \int_0^t \nu^{\ell,\ell'}_{R,k}(s) R^{N,\ell}_k(s)ds \right) +  \sum_{\substack{\ell'=1 \\ \ell'\neq \ell}}^L P^{\ell',\ell}_{R,k} \left( \int_0^t \nu^{\ell',\ell}_{R,k}(s) R^{N,\ell'}_k(s)ds \right), 
\end{split}
\end{equation}
where $P^{\ell,\ell'}_{S,k}, P^{\ell,\ell'}_{I,k}, P^{\ell,\ell'}_{R,k}$, $k\in\sK, \ell, \ell' \in \sL$ are mutually independent standard Poisson processes.

\begin{remark}
The  Poisson processes just introduced above describe the global movements of susceptible, infected and recovered individuals of each group $k$. They are of course not independent of the individual movements $X_{j,k}, Y_{j,k}, Z_{j,k}$. However, we shall not need to specify the dependence between those two collections of random processes.
\end{remark} 
The system of stochastic equations in \eqref{eq:FN}, \eqref{eqn-A-N-jk} and \eqref{eq:SN}--\eqref{eq:RN} uniquely determines the epidemic dynamics.

We define $\bar{S}^{N,\ell}_k(t):=N^{-1}S^{N,\ell}_k(t)$, $\bar{I}^{N,\ell}_k(t):=N^{-1}I^{N,\ell}_k(t)$  and $\bar{R}^{N,\ell}_k(t):=N^{-1}R^{N,\ell}_k(t)$. 
We want to show that under the above assumptions,
we have the following Theorem.
\begin{theorem} \label{thm-LLN}
As $N\to\infty$,
$\big(\bar{S}^{N,\ell}_k, \bar{\mathfrak{F}}_k^{N,\ell}, \bar{I}^{N,\ell}_k, \bar{R}^{N,\ell}_k,\, k \in \sK, \, \ell \in \sL\big)  \to \big(\bar{S}^{\ell}_k, \bar{\mathfrak{F}}_k^\ell, \bar{I}^{\ell}_k, \bar{R}^{\ell}_k,\, k\in \sK, \, \ell \in \sL\big) \qinq \bD^{4KL}$  in probability,  
where 
the limits are  the unique solution of the following system of integral equations: 

\begin{equation}\label{eq:limitLLN}
\begin{split}
 \bar{S}^\ell_k(t)&=\bar{S}^\ell_k(0)-\int_0^t \bar{S}^\ell_k(s)\bar{\Gamma}^\ell_k(s)ds
+\sum_{\ell'=1}^L\int_0^t\nu^{\ell',\ell}_{S,k}(s)\bar{S}^{\ell'}_k(s)ds\,,\\
\bar{\mathfrak{F}}^\ell_{k}(t)&=\bar{\lambda}^0_{k}(t)\sum_{\ell'=1}^L\bar{I}^{\ell'}_k(0)q^{\ell',\ell}_k(0,t)
+\sum_{\ell'}\int_0^t\bar{\lambda}_k(t-s)\bar{S}^{\ell'}_k(s)\bar{\Gamma}^{\ell'}_k(s)q^{\ell',\ell}_k(s,t) ds\,,\\
\bar{I}^{\ell}_k(t) &=  \bar{I}^{\ell}_k(0)  -    \sum_{\ell'=1}^L \bar{I}^{\ell'}_k(0) \int_0^t  q_k^{\ell', \ell}(0,s)  F^0_{k}(ds)   +\int_0^t \bar{S}^\ell_k(s)\bar{\Gamma}^\ell_k(s) ds \\
 & \quad -  \sum_{\ell'=1}^{L} \int_0^t  \int_0^{t-s}   q_k^{\ell', \ell}(s,s+u)  F_k(du)  \bar{S}^{\ell'}_k(s)\bar{\Gamma}^{\ell'}_k(s)  ds  +  \sum_{\ell'=1}^L   \int_0^t  \nu^{\ell',\ell}_{I,k}(s) \bar{I}^{\ell'}_k(s) ds\,, \\
\bar{R}^{\ell}_k(t) &=  \bar{R}^{\ell}_k(0)+  \sum_{\ell'=1}^L \bar{I}^{\ell'}_k(0) \int_0^t  q_k^{\ell', \ell}(0,s)  F^0_{k}(ds)  \\
& \quad +   \sum_{\ell'=1}^{L} \int_0^t  \int_0^{t-s}   q_k^{\ell', \ell}(s,s+u)  F_k(du)  \bar{S}^{\ell'}_k(s)\bar{\Gamma}^{\ell'}_k(s)  ds  +  \sum_{\ell'=1}^L   \int_0^t  \nu^{\ell',\ell}_{R,k} (s)\bar{R}^{\ell'}_k(s)   ds\,,
\end{split}
\end{equation}
where
\begin{equation} \label{eqn-barGamma}
\bar{\Gamma}^\ell_k(t)=(\bar{B}^\ell_k(t))^{-\gamma}\sum_{k'}\sum_{\ell'}\beta^{\ell,\ell'}_{k,k'}(t)\bar{\mathfrak{F}}^{\ell'}_{k'}(t)\,, \quad
\bar{B}_k^{\ell}(t) = \bar{S}^{\ell}_{k}(t) + \bar{I}^{\ell}_{k}(t) + \bar{R}^{\ell}_{k}(t)\,,
\end{equation}
and $ q^{l,l'}_k(s,t) $ is defined in \eqref{transition_function_q}.
\end{theorem}

\begin{remark}
Recall that in the homogenous model, the dynamics of the limits is essentially determined by the two processes $\bar{S}(t)$ and $\bar{\mathfrak{F}}(t)$ with the two integral equations in \eqref{limitmodel}, and then the limits $\bar{I}(t)$ and $\bar{R}(t)$ are given as integral functionals of  $\bar{S}(t)$ and $\bar{\mathfrak{F}}(t)$. This is no longer the case for the multipatch-multigroup model except when $\gamma=0$.  In that case,  $ \bar{S}^\ell_k(t)\bar{\Gamma}^\ell_k(t)= \bar{S}^\ell_k(t)\sum_{k'}\sum_{\ell'}\beta^{\ell,\ell'}_{k,k'}(t)\bar{\mathfrak{F}}^{\ell'}_{k'}(t)$, hence, $( \bar{S}^\ell_k(t), \bar{\mathfrak{F}}^\ell_{k}(t))$ for all $k\in \sK$ and $\ell \in \sL$ can be first determined by  the set of the first two equations in \eqref{eq:limitLLN}, that is,
\begin{equation}\label{eq:limitLLN-gamma0}
\begin{split}
 \bar{S}^\ell_k(t)&=\bar{S}^\ell_k(0)- \int_0^t \bar{S}^\ell_k(s)\sum_{k'}\sum_{\ell'}  \beta^{\ell,\ell'}_{k,k'}(s) \bar{\mathfrak{F}}^{\ell'}_{k'}(s) ds
+\sum_{\ell'=1}^L\int_0^t\nu^{\ell',\ell}_{S,k}(s)\bar{S}^{\ell'}_k(s)ds\,,\\
\bar{\mathfrak{F}}^\ell_{k}(t)&=\bar{\lambda}^0_{k}(t)\sum_{\ell'=1}^L\bar{I}^{\ell'}_k(0)q^{\ell',\ell}_k(0,t)
+\sum_{\ell'}\int_0^t\bar{\lambda}_k(t-s)\bar{S}^{\ell'}_k(s)\sum_{k'}\sum_{\ell''}\beta^{\ell',\ell''}_{k,k'}(s) \bar{\mathfrak{F}}^{\ell''}_{k'}(s)q^{\ell',\ell}_k(s,t) ds\,. 
\end{split}
\end{equation}

However, when $\gamma \in (0,1]$, all the limits $\big(\bar{S}^{\ell}_k, \bar{\mathfrak{F}}_k^\ell, \bar{I}^{\ell}_k, \bar{R}^{\ell}_k,\, k\in \sK, \, \ell \in \sL\big)$ are intertwined in a complicated manner through $\bar{\Gamma}^\ell_k(t)$ in \eqref{eqn-barGamma}. 
Thus, the proof of existence and uniqueness of solution to the system of equations in \eqref{eq:limitLLN} becomes more delicate, see Section \ref{sec-exist-uniq-solution}. 
Moreover, the extension of the new approach from the homogeneous model to the multipatch-multigroup model becomes substantially non-trivial. \end{remark}

\section{Proof of Theorem \ref{thm-LLN}} \label{sec-proofs}

\subsection{Existence and uniqueness of a solution to the system \eqref{eq:limitLLN}} \label{sec-exist-uniq-solution}

\begin{theorem} \label{thm-limit-unique}
The system of equations \eqref{eq:limitLLN} has a unique solution.
\end{theorem}

\begin{proof}
If $\gamma=0$, we have $ \bar{S}^\ell_k(t)\bar{\Gamma}^\ell_k(t)= \bar{S}^\ell_k(t)\sum_{k'}\sum_{\ell'}\beta^{\ell,\ell'}_{k,k'}(t)\bar{\mathfrak{F}}^{\ell'}_{k'}(t)$. The existence and uniqueness of a solution to \eqref{eq:limitLLN-gamma0} follow from a standard argument for Volterra integral equations with Lipschitz coefficients, 
see e.g. Theorems 1.2.13 in \cite{brunner_volterra_2017}.

If $0 < \gamma \le 1$, $ \bar{S}^\ell_k(t)\bar{\Gamma}^\ell_k(t)= \bar{S}^\ell_k(t)(\bar{B}^\ell_k(t))^{-\gamma} \sum_{k'}\sum_{\ell'}\beta^{\ell,\ell'}_{k,k'}(t)\bar{\mathfrak{F}}^{\ell'}_{k'}(t)$, which involves the map $(x,y,z)\mapsto\frac{x}{(x+y+z)^\gamma}$. 
For any $0<\varepsilon<1$, this map is globally Lipschitz on the subset $\{(x,y,z)\in[0,1]^3,\ x+y+z\ge\varepsilon\}$. We will show that for any $T>0$, $\inf_{0\le t\le T}\inf_{\ell,k}\bar{B}^\ell_k(t)>0$.  Then the existence and uniqueness follows from a standard argument of Volterra integral equations with Lipschitz coefficients. 

By the expressions of $ \bar{S}^\ell_k(t)$,  $\bar{I}^\ell_k(t)$ and $ \bar{R}^\ell_k(t)$ in \eqref{eq:limitLLN}, we have, for $ t \in [0,T] $,
\begin{align} \label{eqn-barB-rep-nu}
 \bar{B}^\ell_k(t) & =  \bar{B}^\ell_k(0) + \sum_{\ell'=1}^L\int_0^t\nu^{\ell',\ell}_{S,k}(s)\bar{S}^{\ell'}_k(s)ds +  \sum_{\ell'=1}^L   \int_0^t  \nu^{\ell',\ell}_{I,k}(s) \bar{I}^{\ell'}_k(s) ds +  \sum_{\ell'=1}^L   \int_0^t  \nu^{\ell',\ell}_{R,k} (s)\bar{R}^{\ell'}_k(s)   ds  \non \\
 &=  \bar{B}^\ell_k(0) + \sum_{\ell'\neq \ell}\int_0^t\nu^{\ell',\ell}_{S,k}(s)\bar{S}^{\ell'}_k(s)ds +  \sum_{\ell'\neq\ell}   \int_0^t  \nu^{\ell',\ell}_{I,k}(s) \bar{I}^{\ell'}_k(s) ds +  \sum_{\ell'\neq \ell}   \int_0^t  \nu^{\ell',\ell}_{R,k} (s)\bar{R}^{\ell'}_k(s)   ds \non \\
 & \quad + \int_0^t \nu^{\ell,\ell}_{S,k}(s) \bar{S}^{\ell}_k(s)ds + 
    \int_0^t   \nu^{\ell,\ell}_{I,k}(s)  \bar{I}^{\ell}_k(s) ds +  \int_0^t  \nu^{\ell,\ell}_{R,k} (s) \bar{R}^{\ell}_k(s)   ds \non \\
 & =  \bar{B}^\ell_k(0) +  \underline{\nu}_k^{\ell,\ell}  \int_0^t  \bar{B}^\ell_k(s)  ds +     \int_0^t  \bar V_k^\ell (s) ds\non\\
 &=e^{\underline{\nu}_k^{\ell,\ell} t}\bar{B}^\ell_k(0)+\int_0^te^{\underline{\nu}_k^{\ell,\ell} (t-s)} \bar V_k^\ell (s) ds, 
\end{align}
where $\underline{\nu}_k^{\ell,\ell} := \inf_{s \in [0,T]} \Big\{ \nu^{\ell,\ell}_{S,k}(s) \wedge  \nu^{\ell,\ell}_{I,k}(s) \wedge  \nu^{\ell,\ell}_{R,k}(s)\Big\}$, 
and 
\begin{align} \label{eqn-barV-rep}
 \bar V_k^\ell (s) &=\big( \nu^{\ell,\ell}_{S,k}(s) -\underline{\nu}_k^{\ell,\ell}  \big) \bar{S}^{\ell}_k(s) +
 \big(    \nu^{\ell,\ell}_{I,k}(s) -\underline{\nu}_k^{\ell,\ell} \big) \bar{I}^{\ell}_k(s) + \big(    \nu^{\ell,\ell}_{R,k}(s)-\underline{\nu}_k^{\ell,\ell}   \big) \bar{R}^{\ell}_k(s) \non \\
 & \quad + \sum_{\ell'\neq \ell} \nu^{\ell',\ell}_{S,k}(s)\bar{S}^{\ell'}_k(s) +  \sum_{\ell'\neq}     \nu^{\ell',\ell}_{I,k}(s) \bar{I}^{\ell'}_k(s) +  \sum_{\ell'\neq \ell}   \nu^{\ell',\ell}_{R,k} (s)\bar{R}^{\ell'}_k(s) \non  \\
 & \ge 0 \qforallq s \ge 0. 
\end{align}
Note  that $\nu^{\ell,\ell}_{S,k}(t) = - \sum_{\ell'\neq \ell} \nu^{\ell,\ell'}_{S,k}(t)$ for each $\ell \in \sL$.
We deduce from \eqref{eqn-barB-rep-nu} and \eqref{eqn-barV-rep} (recall that $\underline{\nu}_k^{\ell,\ell}<0$)
\begin{align} \label{eqn-barB-lowerbound}
 \bar{B}^\ell_k(t) \ge  \bar{B}^\ell_k(0) e^{\underline{\nu}_k^{\ell,\ell}   t } \ge \bar{B}^\ell_k(0) e^{ \underline{\nu}_k^{\ell,\ell}   T } := C_{k,T}^\ell \qforallq t \in [0,T]\,.
\end{align}
Moreover our assumption $\inf_{N\ge1}\inf_{\ell,k}N^{-1}B^{N,\ell}_k(0)>0$ a.s. implies that $\inf_{\ell,k}\bar{B}^\ell_k(0)>0$, hence $\inf_{\ell,k}C_{k,T}^\ell>0$. 
\end{proof}

\subsection{An expression of $\bar{S}_k^\ell(t)$ via a Feynman-Kac formula for the associated backward ODEs } \label{sec-barS-FK}

The following estimate will be useful below. From the facts that $\sum_{\ell, k}\bar{I}^\ell_k(t)\le1$ and 
$\lambda(t)\le \lambda^\ast$,  we deduce that for all $t\ge0$,
\begin{equation}\label{est-barF}
\sum_{\ell,k}\bar{\mathfrak{F}}^\ell_{k}(t)\le \lambda^\ast \,.
\end{equation}

From now on, for each $1\le k\le K$, we shall consider the $\{1,\dots,L\}$-valued process $X_k(t)$, which starts,  as specified for the $X_{j,k}$'s at the start of Section \ref{sec:multi},  with the initial distribution $\P(X_k(0)=\ell)=\bar{S}^\ell_k(0)/\bar{S}_k(0) $, $1 \le \ell\le L$.

 The process $X_k(t)$ is a non homogeneous Markov process, whose jumps are specified by the rates
$\nu^{\ell,\ell'}_{S,k}(t)$.

It is clear that the explicit formula $\bar{S}(t)=\bar{S}(0)\exp\left(-\int_0^t \bar{\mathfrak{F}}(s)ds\right)$ in \eqref{eqn-barS-solution} 
was crucial in the proofs of Lemma \ref{le:MKV}  and the convergence of $\bar{I}^N$ in \eqref{eqn-convI-keystep} above. We need an extension of this formula in the context
of the present multipatch-multigroup model. Such a formula will be provided by the next Proposition. 

Let us first rewrite the first line of \eqref{eq:limitLLN}. Denoting by $\mathcal{Q}_k(t)$ the infinitesimal generator of the process $X_k(t)$, i.e., the matrix whose $(\ell,\ell')$ entry
equals $\nu^{\ell,\ell'}_k(t)$, this equation can be rewritten as the following equation for the vector
$\bar{{\bf S}}_k(t):=(\bar{S}^1_k(t),\ldots,\bar{S}^L_k(t))'$ (with $'$ denoting transpose): 
\begin{equation}\label{eq:S} 
\frac{d\bar{{\bf S}}_k}{dt}(t)=-D_k(t)\bar{{\bf S}}_k(t)+\mathcal{Q}'_k(t)\bar{{\bf S}}_k(t)\, ,
\end{equation}
with $D_k(t)$ denoting the diagonal matrix whose $(\ell,\ell)$ entry equals $\bar{\Gamma}^\ell_k(t)$, and $\mathcal{Q}_k'(t)$ denoting the transpose of the matrix $\mathcal{Q}_k(t)$. 
Note that $\bar{{\bf S}}_k(t)$ need not be differentiable at any time $t$, and the right hand side of \eqref{eq:S} can be an irregular function of $t$. The reader can think of \eqref{eq:S} as a short-hand notation for the same equation written in integral form.
We first need to establish a Feynman--Kac formula for a system of backward ODEs adjoint to \eqref{eq:S}. 

\begin{lemma}\label{le:FK}
For any $t>0$, $1\le\ell\le L$, let $\{u_{k,t,\ell}(s),\ 0\le s\le t\}$ be the unique $\R^L$-valued solution of the following backward system of ODEs:
 \begin{equation}\label{eq:backward} 
\frac{d u_{k,t,\ell}}{ds}(s)-D_k(s)u_{k,t,\ell}(s)+\mathcal{Q}_k(s)u_{k,t,\ell}(s)=0\, ,\quad  0\le s\le t\,,
\end{equation}
whose final value $u_{k,t,\ell}(t)$ equals the vector whose $\ell$-th coordinate equals $1$, all others being $0$. 
Then for any  $0\le s<t$, $1\le\ell'\le L$,
\begin{equation}\label{eq:FK}
 u^{\ell'}_{k,t,\ell}(s)=\E\left[{\bf1}_{X_k(t)=\ell}\exp\left(-\int_s^t \bar{\Gamma}^{X_k(r)}_k(r)dr\right)\Big|X_k(s)=\ell'\right]\,,
 \end{equation}
 where $u^{\ell'}_{k,t,\ell}(s)$ stands for the $\ell'$-th coordinate of the vector $u_{k,t,\ell}(s)$.
\end{lemma}

\begin{proof}
In order to simplify our notation, we delete the subindices $k,t,\ell$ and $k$ in the proof. We fix $0\le s<t$. For $s<r\le t$, let $\{u(r),\ s\le r\le t\}$ be the solution of \eqref{eq:backward}. Note that $u\in\bC^1((0,t);\R^L)$.
Let 
\[ V(r) = \exp\left(-\int_s^r\bar{\Gamma}^{X(v)}(v)dv\right),\quad  W(r)=u^{X(r)}(r) V(r)\,.\]
The jump Markov process $X(s)$ has the same law as the solution of the following SDE:
\begin{align*}
X(t)&=X(s)+\sum_{\ell,\ell'}(\ell'-\ell)\int_s^t\int_0^\infty{\bf1}_{X(r^-)=\ell}{\bf1}_{v\le\nu_S^{\ell,\ell'}(r)}Q^{\ell,\ell'}(dr,dv)\\
&=X(s)+\sum_{\ell'}\int_s^r (\ell'-X(v))\nu_S^{X(v),\ell'}(v)dv+
\sum_{\ell,\ell'}(\ell'-\ell)\int_s^r\int_0^\infty{\bf1}_{X(r^-)=\ell}{\bf1}_{v\le\nu_S^{\ell,\ell'}(r)}\bar{Q}^{\ell,\ell'}(dr,dv)\,,
\end{align*}
where $\{Q^{\ell,\ell'},\ 1\le\ell,\ell'\le L\}$ are mutually independent standard PRMs on $\R^2$, and $\bar{Q}^{\ell,\ell'}(dr,dv)=Q^{\ell,\ell'}(dr,dv)-drdv$. From this it follows that
\begin{align*}
W(t)&=W(s)+\int_s^t \left[\frac{du^{X(r)}(r)}{dr}-\bar{\Gamma}^{X(r)}(r)u^{X(r)}(r)\right]V(r)dr \\
& \qquad \quad  +\sum_{\ell}\int_s^t\nu_S^{X(r),\ell}(r)\big[u^{\ell}(r)-u^{X(r)}(r)\big]V(r)dr+M(t)\,,
\end{align*}
where $M(t)$ is a martingale such that $M(s)=0$. We further note that
\begin{align*}
\sum_\ell (u^\ell-u^x)\nu^{x,\ell}(v)&=\sum_{\ell\not=x}(u^\ell-u^x)\nu^{x,\ell}(v)\\
&=(\mathcal{Q}u)^x\,.
\end{align*}
From the above formulas, we deduce that
\begin{align*}
W(t)&=W(s)+\int_s^t\left[\frac{du^{X(r)}(r)}{dr}-(Du)^{X(r)}(r)+(\mathcal{Q}u)^{X(r)}(r)\right]V(r)dr+M(t)\\
&=W(s)+M(t)\,,
\end{align*}
where we have used the fact that $u$ solves \eqref{eq:backward}.
Taking the conditional expectation given that  $X(s)=\ell'$ in the last identity yields the formula \eqref{eq:FK},
since 
\begin{align*}
W(t)&={\bf1}_{X(t)=\ell}\exp\left(-\int_s^t\bar{\Gamma}^{X(v)}(v)dv\right),\ \text{ and}\\
W(s)&=u^{X(s)}(s)\,.
\end{align*}
\end{proof}

\begin{remark}
Would the Markov process $X_k$  be a diffusion, then the system of ODEs \eqref{eq:backward} would be replaced by a parabolic PDE, and a similar Feynman--Kac formula in such a situation is well-known, see, e.g., Chapter 3.8 in \cite{PaRa}.
\end{remark}

We can now derive an explicit formula for $\bar{S}^\ell_k(t)$.
\begin{prop}\label{pro:FK}
The solution of equation \eqref{eq:S} is given by the following formula: for any $1\le k\le K$,
$1\le\ell\le L$, $t>0$,
\begin{equation}\label{eq:FKr}
\bar{S}^\ell_k(t)=\bar{S}_k(0)\E\left[{\bf1}_{X_k(t)=\ell}\exp\left(-\int_0^t \bar{\Gamma}^{X_k(s)}_k(s)ds\right)\right],
\end{equation} 
where $X_k$ is initialized as indicated above. 
\end{prop}

\begin{proof}
The duality between equations \eqref{eq:S} and \eqref{eq:backward} is expressed by the obvious fact that
\[ \frac{d}{ds}(\bar{{\bf S}}_k(s),u_{k,t,\ell}(s))=0,\quad 0\le s\le t,\] 
hence $(\bar{{\bf S}}_k(t),u_{k,t,\ell}(t))=(\bar{{\bf S}}_k(0),u_{k,t,\ell}(0))$. 
Here we use $(\cdot, \cdot)$ as the scalar product in $\R^L$.
Recall that by our choice of $u_{k,t,\ell}(t)$, $(\bar{{\bf S}}_k(t),u_{k,t,\ell}(t))=\bar{S}^\ell_k(t)$.
Now we deduce from \eqref{eq:FK} with $s=0$ that
\begin{align*}
 (\bar{{\bf S}}_k(0),u_{k,t,\ell}(0))&=\sum_{\ell'}\bar{S}_k^{\ell'}(0)
\E\left[{\bf1}_{X_k(t)=\ell}\exp\left(-\int_0^t \bar{\Gamma}^{X_k(r)}_k(r)dr\right)\Big|X_k(0)=\ell'\right]\\
&=\bar{S}_k(0)\sum_{\ell'}
\E\left[{\bf1}_{X_k(t)=\ell}\exp\left(-\int_0^t \bar{\Gamma}^{X_k(r)}_k(r)dr\right)\Big|X_k(0)=\ell'\right]\P(X_k(0)=\ell')\\
&=\bar{S}_k(0)\E\left[{\bf1}_{X_k(t)=\ell}\exp\left(-\int_0^t \bar{\Gamma}^{X_k(r)}_k(r)dr\right)\right]\,.
\end{align*}
The result follows from the last three identities. 
\end{proof}

\subsection{An auxiliary system of Poisson-driven stochastic equations} \label{sec-McKV-SDE}

We want to associate to a collection of mutually independent standard PRMs $\{Q^\ell_{k},\  k \in \sK, \ell \in \sL\}$ on $\R^2_+$ and a family of mutually independent processes $ \{X_k(t), t \geq 0, k \in \sK\} $, also independent from the $ \lbrace Q^\ell_k, k \in \sK, \ell \in \sL \}$, the processes $\{A^\ell_{k}(t),\  k \in \sK, \ell \in \sL\}$, which is  the solution of the following system of stochastic equations: 

\begin{equation}\label{eq:McKV}
\begin{split}
A^\ell_{k}(t)&=\int_0^t\int_0^\infty{\bf1}_{A_{k}(s^-)=0}{\bf1}_{X_k(s)=\ell}{\bf1}_{u\le  \bar{\Gamma}^{\ell}_k(s^-)}   Q^\ell_k(ds,du)\,,\quad \text{ with } A_{k}(t)=\sum_\ell A^\ell_k(t)\,,\\
\bar{\mathfrak{G}}^{\ell}_{k}(t)&=\bar{\lambda}^0_{k}(t)\sum_{\ell'=1}^L\bar{I}^{\ell'}_k(0)q^{\ell',\ell}_k(0,t)
+\bar{S}_k(0)\E\left[{\lambda}_k(t-\tau_{k})q^{X_k(\tau_{k}),\ell}_k(\tau_{k},t)\right]
\,,\\
 \bar{S}^\ell_k(t)&=\bar{S}^\ell_k(0)-\int_0^t \bar{S}^\ell_k(s)\bar{\Gamma}^\ell_k(s)ds
+\sum_{\ell'=1}^L\int_0^t\nu^{\ell',\ell}_{S,k}(s)\bar{S}^{\ell'}_k(s)ds\,, \\
\bar{I}^{\ell}_k(t) &=  \bar{I}^{\ell}_k(0)  - \sum_{\ell'=1}^L \bar{I}^{\ell'}_k(0) 
 \int_0^t  q_k^{\ell', \ell}(0,s)  F_{k}^0(ds)   + \int_0^t \bar{S}^\ell_k(s)\bar{\Gamma}^\ell_k(s) ds \\
 & \quad -   \sum_{\ell'=1}^{L} \int_0^t  \int_0^{t-s}   q_k^{\ell', \ell}(s,s+u)  F_k(du)  \bar{S}^{\ell'}_k(s)\bar{\Gamma}^{\ell'}_k(s)  ds  +  \sum_{\ell'=1}^L   \int_0^t  \nu^{\ell',\ell}_{I,k}(s) \bar{I}^{\ell'}_k(s) ds\,, \\
\bar{R}^{\ell}_k(t) &=    \sum_{\ell'=1}^L \bar{I}^{\ell'}_k(0) \int_0^t  q_k^{\ell', \ell}(0,s)  F^0_{k}(ds) +   \sum_{\ell'=1}^{L} \int_0^t  \int_0^{t-s}   q_k^{\ell', \ell}(s,s+u)  F_k(du)  \bar{S}^{\ell'}_k(s)\bar{\Gamma}^{\ell'}_k(s)  ds \\
& \quad +  \sum_{\ell'=1}^L   \int_0^t  \nu^{\ell',\ell}_{R,k} (s)\bar{R}^{\ell'}_k(s)   ds\,,
\end{split}
\end{equation} 
where
\begin{align*}
 \bar{\Gamma}^\ell_k(t) &=(\bar{B}^\ell_k(t))^{-\gamma}\sum_{k'}\sum_{\ell'}\beta^{\ell,\ell'}_{k,k'}(t)\bar{\mathfrak{G}}^{\ell'}_{k'}(t)\,, \\
 \tau_{k}&=\inf\{t>0,\ A_{k}(t)=1\}\,,  \\
\bar{B}^\ell_k(t)&= \bar{S}^\ell_k(t)+ \bar{I}^\ell_k(t)+ \bar{R}^\ell_k(t)\,.
\end{align*}

\begin{remark} \label{rem-McKV-mpmg}
In this system of stochastic equations, the laws of the random functions $X_k$ and $\lambda_k$ as well as that of the PRMs
$Q^\ell_k$  are given. However, $\tau_k$ is an unknown of this equation, whose law enters the coefficient on the second line. 
Recall Remark \ref{rem-McKV-hom}. We can thus regard these equations as a McKean--Vlasov stochastic equations, and also the following as a propagation of chaos result for the times of infection of the various initially susceptible individuals.  
\end{remark}

\begin{remark}
In the case $\gamma=0$, one can instead define the following simpler system of Poisson-driven SDEs:
\begin{equation*}\label{eq:McKV-gamma0}
\begin{split}
A^\ell_{k}(t)&=\int_0^t\int_0^\infty{\bf1}_{A_{k}(s^-)=0}{\bf1}_{X_k(s)=\ell}{\bf1}_{u\le  \bar{\Gamma}^{\ell}_k(s^-)}   Q^\ell_k(ds,du)\,,\quad \text{ with } A_{k}(t)=\sum_\ell A^\ell_k(t)\,,\\
\bar{\mathfrak{G}}^{\ell}_{k}(t)&=\bar{\lambda}^0_{k}(t)\sum_{\ell'=1}^L\bar{I}^{\ell'}_k(0)q^{\ell',\ell}_k(0,t)
+\bar{S}_k(0)\E\left[{\lambda}_k(t-\tau_{k})q^{X_k(\tau_{k}),\ell}_k(\tau_{k},t)\right]
\,,\\
\end{split}
\end{equation*} 
where
\begin{align*}
 \bar{\Gamma}^\ell_k(t) &=\sum_{k'}\sum_{\ell'}\beta^{\ell,\ell'}_{k,k'}(t)\bar{\mathfrak{G}}^{\ell'}_{k'}(t)\,, \\
 \tau_{k}&=\inf\{t>0,\ A_{k}(t)=1\}\,. 
\end{align*}
The proofs below can be simplified in this case, see also Remark \ref{rem-proof-gamma0}. 
\end{remark}

We first need to show:
\begin{lemma}\label{le:EU}
The system of equations \eqref{eq:McKV} has a unique solution, which is such that 
$\bar{\mathfrak{G}}^{\ell}_{k}\equiv\bar{\mathfrak{F}}^{\ell}_{k}$, for all $k \in \sK, \ell \in \sL$.
\end{lemma} 
\begin{proof}
To any $\{m^{\ell}_{k},\ k \in \sK, \ell \in \sL\}\in \bD^{LK}$ which satisfies $\inf_{\ell,k,0\leq t \leq T}m^{\ell}_{k}(t)\ge0$
and $\sup_{0\le t\le T,\ell,k}m^{\ell}_{k}(t)\le\lambda^\ast$ for all $T$, we associate $\{A^{(m)}_{k},\ 1\le k\le K\}$,
which solves
\begin{align*}
A^{(m)}_{\ell,k}(t)=\int_0^t\int_0^\infty{\bf1}_{A^{(m)}_{k}(s^-)=0}{\bf1}_{X_k(s)=\ell}{\bf1}_{u\le\bar{\Gamma}^{(m)}_{\ell,k}(s^-)}Q^\ell_{k}(ds,du),\ \text{ with }A^{(m)}_{k}(t)=\sum_\ell A^{(m)}_{\ell,k}(t),
\end{align*}
where $X_k(s)$ is as above a $\{0,1,\ldots,L\}$-valued Markov jump process
which is such that $\P(X_k(0)=\ell)=\bar{S}^\ell_k(0) $, $1\le\ell\le L$, $\P(X_k(0)=0)=1-\sum_{\ell'}\bar{S}^{\ell'}_k(0)$, and starting from $0$, $X_k(t)=0$ for all $t>0$, and 
\[\bar{\Gamma}^{(m)}_{\ell,k}(t)=\big(\bar{B}^{(m)}_{\ell,k}(t)\big)^{-\gamma}\sum_{k',\ell'}\beta^{\ell,\ell'}_{k,k'}(t)m^{\ell'}_{k'}(t)\,,\]
with $\bar{B}^{(m)}_{\ell,k}(t)=\bar{S}^{(m)}_{\ell,k}(t)+\bar{I}^{(m)}_{\ell,k}(t)+\bar{R}^{(m)}_{\ell,k}(t)$, where
$(\bar{S}^{(m)}_{\ell,k}(t),\bar{I}^{(m)}_{\ell,k}(t),\bar{R}^{(m)}_{\ell,k}(t))$ solves  the last three lines of 
\eqref{eq:McKV}, with $\bar{\Gamma}^\ell_k$ replaced by $\bar{\Gamma}^{(m)}_{\ell,k}$. 
Note that Theorem 
\ref{thm-limit-unique} applies to this system of equations. In particular $\bar{B}^{(m)}_{\ell,k}(t)$ satisfies clearly the lower bound which we have established for $\bar{B}^\ell_k(t)$. 
We moreover define $\tau^{(m)}_{k}=\inf\{t>0,\  A^{(m)}_{k}(t)=1\}$. 

The result will follow from the existence and uniqueness of $m^\ast:=\big\{m^{\ast,\ell}_{k},\ k \in \sK, \ell \in \sL\big\}$ such that $m^\ast=\bar{\mathfrak{G}}^{(m^\ast)}$, where
\begin{align*}
\bar{\mathfrak{G}}^{(m)}_{\ell,k}(t)&=\bar{\lambda}^0_{k}(t)\sum_{\ell'=1}^L\bar{I}^{\ell'}_k(0)q^{\ell',\ell}_k(0,t)
+\bar{S}_k(0)\E\left[{\lambda}_k(t-\tau^{(m)}_{k})
q^{X_k(\tau^{(m)}_{k}),\ell}_k(\tau^{(m)}_{k},t)\right]\\
&=\bar{\lambda}^0_{k}(t)\sum_{\ell'=1}^L\bar{I}^{\ell'}_k(0)q^{\ell',\ell}_k(0,t) 
+\bar{S}_k(0)\E\left[\int_0^t{\lambda}_k(t-s)q^{X_k(s),\ell}_k(s,t)dA^{(m)}_{k}(s)\right]\\
&=\bar{\lambda}^0_{k}(t)\sum_{\ell'=1}^L\bar{I}^{\ell'}_k(0)q^{\ell',\ell}_k(0,t) \\
& \quad +\bar{S}_k(0)\E\left[\int_0^t{\lambda}_k(t-s)q^{X_k(s),\ell}_k(s,t)\P(A^{(m)}_{k}(s)=0|X_k)\bar{\Gamma}^{(m)}_{X_k(s),k}(s)ds\right]\\
&=\bar{\lambda}^0_{k}(t)\sum_{\ell'=1}^L\bar{I}^{\ell'}_k(0)q^{\ell',\ell}_k(0,t)\\
&\quad
+\bar{S}_k(0)\E\left[\int_0^t\!{\lambda}_k(t-s)q^{X_k(s),\ell}_k(s,t)\bar{\Gamma}^{(m)}_{X_k(s),k}(s)
\exp\!\left(\!\!-\!\!\int_0^s\!\!\bar{\Gamma}^{(m)}_{X_k(r), k}(r)dr\!\!\right)ds\right]\\
&=\bar{\lambda}^0_{k}(t)\sum_{\ell'=1}^L\bar{I}^{\ell'}_k(0)q^{\ell',\ell}_k(0,t)\\
&\quad+\bar{S}_k(0)\sum_{\ell'=1}^L\E\left[\!\int_0^t\!{\bf1}_{X_k(s)=\ell'}{\lambda}_k(t-s)q^{\ell',\ell}_k(s,t)\bar{\Gamma}^{(m)}_{\ell',k}(s)
\exp\!\left(\!\!-\!\!\int_0^s\!\!\bar{\Gamma}^{(m)}_{X_k(r), k}(r)dr\!\!\right)ds\right]\\
&=\bar{\lambda}^0_{k}(t)\sum_{\ell'=1}^L\bar{I}^{\ell'}_k(0)q^{\ell',\ell}_k(0,t)
+\sum_{\ell'=1}^L\int_0^t \bar{S}^{(m),\ell'}_{k}(s)\bar{\lambda}_k(t-s)q^{\ell',\ell}_k(s,t)\bar{\Gamma}^{(m)}_{\ell',k}(s)ds\,,
\end{align*}
where in the last equality we have defined 
\[ \bar{S}^{(m),\ell'}_{k}(t):=\bar{S}_k(0)\E\left[{\bf1}_{X_k(t)=\ell'}\exp\left(-\int_0^t \bar{\Gamma}^{(m)}_{X_k(s), k}(s)ds\right)\right]\,.\]
It follows from \eqref{eq:limitLLN} and \eqref{eq:FKr} that $m=\bar{\mathfrak{G}}^{(m)}$ iff 
$(\bar{S}^{(m)},\bar{\mathfrak{G}}^{(m)},\bar{I}^{(m)},\bar{R}^{(m)})$ solves \eqref{eq:limitLLN}. Since that system of integral equations has a unique solution, the equation
$m=\bar{\mathfrak{G}}^{(m)}$ has a unique solution $m^\ast$, 
and moreover $m^{\ast,\ell}_k\equiv\bar{\mathfrak{F}}^\ell_k$ for all $\ell$, $k$.
\end{proof}

\subsection{Estimates using the i.i.d. processes constructed from the Poisson-driven stochastic equations \eqref{eq:McKV} } \label{sec-estimates-iid}

Recall $A_{j,k}^{N,\ell}(t)$ in \eqref{eqn-A-N-jk}. 
Let $\big\{A_{j,k}^\ell(t): j \ge 1\big\}$ be the solution of  \eqref{eq:McKV} with $(Q^\ell_{k}, X_{k}, \lambda_k)$ replaced by $(Q^\ell_{j,k}, X_{j,k}, \lambda_{j,k})$ for each $j \ge 1$ (which enter in the definition of  $A^{N,\ell}_{j,k}$). 
We need the following lemma on the approximation of $A^{N,\ell}_{j,k}$ by $A^\ell_{j,k}$ as in Lemma \ref{le:estim}. 
In the sequel, we shall use the notation
\[ \mathcal{E}^N_T:=\left\{\inf_{0\le t\le T, 1\le k\le K, 1\le \ell\le L}\bar{B}^{N,\ell}_k(t)\ge C_{T}^\ast\right\}\,,\]
where  $C^\ast_T:=\frac{1}{4}\min_{k,\ell} C^\ell_{k,T}$, and the $C^\ell_{k,T}$'s are the lower bounds which appear in formula \eqref{eqn-barB-lowerbound}. 
Lemma~\ref{le:BN>c} below establishes that $\P\left((\mathcal{E}^N_T)^c\right)\to 0$ as $N\to\infty$. 
Moreover, we define for any $T>0$ the stopping time
\begin{equation*}\label{stopB}
\tT:=\inf\Big\{t>0,\ \inf_{k,\ell}\bar{B}^{N,\ell}_k(t)< C^\ast_T\Big\}\,.
\end{equation*}
Note that on the event  $\mathcal{E}^N_T$, $\tT\ge T$.
\begin{lemma} \label{lem-An-Ai-diffbound}
For any $T>0$, $k\in\sK$, and $\ell\in\sL$, as $N\to\infty$,
\begin{equation*}
\frac{1}{N} \E \left[\sum_{j=1}^{S^N_k(0)}  \sup_{0\le t \le T\wedge\tT} \Big|A^{N,\ell}_{j,k}(t) - A^{\ell}_{j,k}(t) \Big|\right]  \to 0\,.
\end{equation*}
\end{lemma}

Before we prove this Lemma, let us first understand how to bound $\bar{B}^{N,\ell}_k(t)$ from below.

\begin{lemma}\label{le:BN>c}
For any $T>0$,  $k\in\sK$ and $\ell\in\sL$, 
as $N\to\infty$,
\[ \P\left(\inf_{0\le t\le T}\bar{B}^{N,\ell}_k(t)< C^\ast_T \right)\to 0\,,\]
where  $C^\ast_T$ is specified above in the definition of $\mathcal{E}^N_T$. 
\end{lemma}
Note that Lemma \ref{lem-An-Ai-diffbound} and Lemma \ref{le:BN>c} imply that, as
$N\to\infty$,
\[ \frac{1}{N}\sum_{j=1}^{S^N_k(0)}  \sup_{0\le t \le T} \Big|A^{N,\ell}_{j,k}(t) - A^{\ell}_{j,k}(t) \Big|\to0\,,\]
in probability.
\begin{proof}[Proof of Lemma~\ref{le:BN>c}]
By \eqref{eq:SN}, \eqref{eq:IN} and \eqref{eq:RN}, we obtain,
with $\bar{P}^{\ell,\ell'}_{S,k}, \bar{P}^{\ell,\ell'}_{I,k}, \bar{P}^{\ell,\ell'}_{R,k}$, $k\in\sK, \ell, \ell' \in \sL$ the corresponding compensated standard Poisson processes,
\begin{align*}
\bar{B}^{N,\ell}_k(t)
&= \bar{B}^{N,\ell}_k(0)  - \sum_{\ell'=1,  \ell'\neq \ell}^L  \int_0^t\nu^{\ell,\ell'}_{S,k}(s)\bar{S}^{N,\ell}_k(s)ds + \sum_{\ell'=1, \ell'\neq \ell}^L\int_0^t\nu^{\ell',\ell}_{S,k}(s)\bar{S}^{N,\ell'}_k(s)ds  \\
& \quad  - \sum_{\ell'=1, \ell'\neq \ell}^L \int_0^t\nu^{\ell,\ell'}_{I,k}(s)I^{N,\ell}_k(s)ds +  \sum_{\ell'=1, \ell'\neq \ell}^L   \int_0^t  \nu^{\ell',\ell}_{I,k}(s) \bar{I}^{N,\ell'}_k(s) ds  \\
& \quad -  \sum_{\ell'=1, \ell'\neq \ell}^L   \int_0^t \nu^{\ell,\ell'}_{R,k}(s) R^{N,\ell}_k(s)ds  +  \sum_{\ell'=1,  \ell'\neq \ell}^L   \int_0^t  \nu^{\ell',\ell}_{R,k} (s)\bar{R}^{N,\ell'}_k(s)   ds \\
& \quad - N^{-1}\sum_{\ell'=1,  \ell'\neq \ell}^L \bar{P}^{\ell,\ell'}_{S,k}\left(\int_0^t\nu^{\ell,\ell'}_{S,k}(s)S^{N,\ell}_k(s)ds\right) 
+ N^{-1}\sum_{\ell'=1, \ell'\neq \ell}^L \bar{P}^{\ell',\ell}_{S,k}\left(\int_0^t\nu^{\ell',\ell}_{S,k}(s)S^{N,\ell'}_k(s)ds\right) \\
&\quad - N^{-1} \sum_{\ell'=1, \ell'\neq \ell}^L \bar{P}^{\ell,\ell'}_{I,k}\left(\int_0^t\nu^{\ell,\ell'}_{I,k}(s)I^{N,\ell}_k(s)ds\right) 
+ N^{-1}\sum_{\ell'=1, \ell'\neq \ell}^L \bar{P}^{\ell',\ell}_{I,k}\left(\int_0^t\nu^{\ell',\ell}_{I,k}(s)I^{N,\ell'}_k(s)ds\right) \\
& \quad - N^{-1} \sum_{\ell'=1, \ell'\neq \ell}^L \bar{P}^{\ell,\ell'}_{R,k} \left( \int_0^t \nu^{\ell,\ell'}_{R,k}(s) R^{N,\ell}_k(s)ds \right) + N^{-1} \sum_{\ell'=1,  \ell'\neq \ell}^L \bar{P}^{\ell',\ell}_{R,k} \left( \int_0^t \nu^{\ell',\ell}_{R,k}(s) R^{N,\ell'}_k(s)ds \right) \\
& =  \bar{B}^{N,\ell}_k(0) + \int_0^t\underline{\nu}^{\ell,\ell}_k\bar{B}^{N,\ell}_k(s)ds+ \int_0^t Z^{N,\ell}_k(s) ds +\xi^{N,\ell}_k(t) \\
&=e^{\underline{\nu}^{\ell,\ell}_k t}\bar{B}^{N,\ell}_k(0) + 
\int_0^t e^{\underline{\nu}^{\ell,\ell}_k(t-s)}Z^{N,\ell}_k(s)ds+
e^{\underline{\nu}^{\ell,\ell}_k t}\int_0^te^{-\underline{\nu}^{\ell,\ell}_k s}d\xi^{N,\ell}_k(s)\,,
\end{align*}
where $ \underline{\nu}^{\ell,\ell}_k$ has been defined in the proof of Theorem \ref{thm-limit-unique},   
 \begin{align*}
 Z^{N,\ell}_k(s) &=  \sum_{\ell'\neq \ell}\nu^{\ell',\ell}_{S,k}(s) \bar{S}^{N,\ell'}_k(s)+   \sum_{\ell'\neq \ell}\nu^{\ell',\ell}_{I,k}(s)\bar{I}^{N,\ell'}_k(s)+\sum_{\ell'\neq \ell}\nu^{\ell',\ell}_{R,k}(s)\bar{R}^{N,\ell'}_k(s) \\
 &\quad+(\nu^{\ell,\ell}_{S,k}(s)-\underline{\nu}^{\ell,\ell}_k)\bar{S}^{N,\ell}_k(s)+
  (\nu^{\ell,\ell}_{I,k}(s) - \underline{\nu}^{\ell,\ell}_k )  \bar{I}^{N,\ell}_k(s) + 
  (\nu^{\ell,\ell}_{R,k}(s)-  \underline{\nu}^{\ell,\ell}_k)  \bar{R}^{N,\ell}_k(s)   \\
 &\ge0\, ,
 \end{align*} 
 and $\xi^{N,\ell}_k(t)$ is the sum of the last six terms with compensated Poisson processes in the first right hand side. Observe that $\xi^{N,\ell}_k(t)$ is a square integrable martingale with respect to a filtration $\sF^N_t$ which we will specify next, with a quadratic variation which is bounded by $c_T/N$ for $0\le t\le T$. 
 For each $t\ge0$, the $\sigma$--algebra $\sF^N_t$ is generated by $S^{N,\ell}_k(s), I^{N,\ell}_k(s), R^{N,\ell}_k(s)$ for all 
 $\ell\in\sL$, $k\in\sK$ and $0\le s\le t$.
 We deduce from the above computations that 
\[ \bar{B}^{N,\ell}_k(t)\ge e^{\underline{\nu}^{\ell,\ell}_k t}\left[\bar{B}^{N,\ell}_k(0) +
\int_0^te^{-\underline{\nu}^{\ell,\ell}_k s}d\xi^{N,\ell}_k(s)\right]\,.\]
We define the martingale
\[ \zeta^{N,\ell}_k(s):=\int_0^te^{-\underline{\nu}^{\ell,\ell}_ks}d\xi^{N,\ell}_k(s)\,,\]
whose quadratic variation is again bounded by $c_T/N$ for $0\le t\le T$. We define the events 
\begin{align*}
 \mathcal{A}^{N,\ell}_k&:=\left\{\zeta^{N,\ell}_k(t)\ge-\frac{1}{2}\bar{B}^{N,\ell}_k(0) ,\ \forall 0\le t\le T\right\}\,.
\end{align*}
It is easy to verify that on the event $\mathcal{A}^{N,\ell}_k$, for $0\le t\le T$,
\[ \bar{B}^{N,\ell}_k(t)\ge\frac{1}{2}\bar{B}^{N,\ell}_k(0) \exp\left(\underline{\nu}^{\ell,\ell}_k T\right),\]
hence  $\cap_{k,\ell}\mathcal{A}^{N,\ell}_k\subset\mathcal{E}^N_T$,
while $\P(\mathcal{A}^{N,\ell}_k)\ge 1-c'_T/N$. 
Since  
$\P\big(\inf_{\ell,k}\bar{B}^{N,\ell}_k(0)\ge\frac{1}{2}\inf_{\ell,k}\bar{B}^{\ell}_k(0)\big)\to1$  as $N\to\infty$, the result follows.
\end{proof}

\begin{proof}[Proof of Lemma \ref{lem-An-Ai-diffbound}] 
In this proof, $C$ will denote an arbitrary positive constant, and $\ep_N$ an arbitrary sequence of positive numbers which converges to $0$ as $N\to\infty$. Both $C$ and $\ep_N$ may vary from one line to another.
Then,
\begin{align*}
& \Big|A^{N,\ell}_{j,k}(t)  - A^\ell_{j,k}(t) \Big| 
 \le \int_0^t \int_{\bar\Gamma^{N,\ell}_{k}(s^-)\wedge \bar\Gamma^{\ell}_{k}(s)}^ {\bar\Gamma^{N,\ell}_{k}(s^-)\vee \bar\Gamma^{\ell}_{k}(s)} Q^\ell_{j,k}(ds,du)\,, 
\end{align*}
and
\begin{align*}
	\sup_{0 \le r \le t}  \Big|A^{N,\ell}_{j,k}(r)  - A^\ell_{j,k}(r) \Big| \le  \int_0^t \int_{\bar\Gamma^{N,\ell}_{k}(s^-)\wedge \bar\Gamma^{\ell}_{k}(s)}^ {\bar\Gamma^{N,\ell}_{k}(s^-)\vee \bar\Gamma^{\ell}_{k}(s)} Q^\ell_{j,k}(ds,du)\,,
\end{align*}
from which we obtain
\begin{align} \label{eqn-A-jk-sup-bound}
 \E\bigg[ \sup_{0 \le r \le t\wedge\tT}  \Big|A^{N,\ell}_{j,k}(r)  - A^\ell_{j,k}(r) \Big| \bigg]  \le  \E \left[\int_0^{t\wedge\tT}   \big|\bar\Gamma^{N,\ell}_{k}(s) - \bar\Gamma^{\ell}_{k}(s)\big|   ds \right] \,. 
\end{align}
 We have just used the well-known fact that the expectation of the integral of a predictable process w.r.t. a PRM equals the expectation of the integral of the same process w.r.t. the mean measure of the PRM.
We now have
\begin{align} \label{eqn-Gamma-diff-bound}
\big|\bar\Gamma^{N,\ell}_{k}(t) - \bar\Gamma^{\ell}_{k}(t)\big|  
& =\Bigg|  \frac{1}{(\bar{B}^{N,\ell}_k(t))^\gamma}
\sum_{k'}\sum_{\ell'}\beta^{\ell,\ell'}_{k,k'}(t) \bar{\mathfrak{F}}^{N,\ell'}_{k'}(t) -  \frac{1}{(\bar{B}^{\ell}_k(t))^\gamma}
\sum_{k'}\sum_{\ell'}\beta^{\ell,\ell'}_{k,k'}(t) \bar{\mathfrak{F}}^{\ell'}_{k'}(t) \Bigg|   \non  \\
& \le   \frac{1}{(\bar{B}^{\ell}_k(t))^\gamma}
\sum_{k'}\sum_{\ell'}\beta^{\ell,\ell'}_{k,k'}(t)  \Big| \bar{\mathfrak{F}}^{N,\ell'}_{k'}(t)  - \bar{\mathfrak{F}}^{\ell'}_{k'}(t) \Big|   \non \\
& \quad + 
 \Bigg|  \frac{1}{(\bar{B}^{N,\ell}_k(t))^\gamma}-  \frac{1}{(\bar{B}^{\ell}_k(t))^\gamma} \Bigg|  
\sum_{k'}\sum_{\ell'}\beta^{\ell,\ell'}_{k,k'}(t) \bar{\mathfrak{F}}^{N,\ell'}_{k'}(t)  \non  \\
& \le C_T^{-\gamma}  \beta^*\sum_{\ell',k'}   \Big| \bar{\mathfrak{F}}^{N,\ell'}_{k'}(t)  - \bar{\mathfrak{F}}^{\ell'}_{k'}(t) \Big| + \lambda^* \beta^*LK  \Bigg|  \frac{1}{(\bar{B}^{N,\ell}_k(t))^\gamma}-  \frac{1}{(\bar{B}^{\ell}_k(t))^\gamma} \Bigg|\,,  
\end{align}
where we have used for the last inequality both the lower bound \eqref{eqn-barB-lowerbound}, and the fact that
$\bar{\mathfrak{F}}^{N,\ell}_{k}(t) \le \lambda^* (\bar{I}^N_k(0) + \bar{S}^N_k(0) ) \le \lambda^*$ for all $t\ge 0$ and all $k, \ell$, which follows from \eqref{eq:FN}.

By  \eqref{eq:FN} and \eqref{eq:limitLLN}, we obtain
\begin{equation}\label{conv:FN}
\begin{split}
  {\bf1}_{t<\tT}\Big| \bar{\mathfrak{F}}^{N,\ell}_{k}(t)  - \bar{\mathfrak{F}}^{\ell}_{k}(t) \Big|
  & \le \Bigg|  N^{-1}\sum_{j=1}^{I^{N}_k(0)}\lambda_{-j,k}(t){\bf1}_{Y_{j,k}(t)=\ell} - \bar{\lambda}^0_{k}(t)\sum_{\ell'=1}^L\bar{I}^{\ell'}_k(0)q^{\ell',\ell}_k(0,t)\Bigg| \\
  & \quad + {\bf1}_{t<\tT}\Bigg| N^{-1} \sum_{j=1}^{S^{N}_k(0)}\lambda_{j,k}(t-\tau_{j,k}^N)\sum_{\ell'} {\bf1}_{X_{j,k}(\tau_{j,k}^N)=\ell'}{\bf1}_{Y^{\tau_{j,k}^N,\ell'}_{j,k}(t) =\ell}  \\
  & \qquad \quad  -\sum_{\ell'}\int_0^t\bar{\lambda}_k(t-s)\bar{S}^{\ell'}_k(s)\bar{\Gamma}^{\ell'}_k(s)q^{\ell',\ell}_k(s,t) ds  \Bigg|\,. 
\end{split}
\end{equation}
The convergence to $0$ in $ L^1 $ of the first term follows from the law of large numbers.
 Concerning the second term, we first note that
 \begin{align}\label{decomp-FN}
\frac{1}{N} \sum_{j=1}^{S^{N}_k(0)}&\lambda_{j,k}(t-\tau_{j,k}^N)\sum_{\ell'} {\bf1}_{X_{j,k}(\tau_{j,k}^N)=\ell'}{\bf1}_{Y^{\tau_{j,k}^N,\ell'}_{j,k}(t) =\ell}\nonumber
\\&=
\frac{1}{N} \sum_{j=1}^{S^{N}_k(0)}\int_0^t\lambda_{j,k}(t-s)\sum_{\ell'} {\bf1}_{X_{j,k}(s)=\ell'}
{\bf1}_{Y^{s,\ell'}_{j,k}(t) =\ell}d A^{N,\ell'}_{j,k}(s)\nonumber\\
&=\frac{1}{N} \sum_{j=1}^{S^{N}_k(0)}\int_0^t\int_0^\infty\lambda_{j,k}(t-s)\sum_{\ell'} 
{\bf1}_{A^N_{j,k}(s^-)=0}{\bf1}_{X_{j,k}(s)=\ell'}
{\bf1}_{Y^{s,\ell'}_{j,k}(t) =\ell}{\bf1}_{u\le\bar{\Gamma}^{N,\ell'}_k(s^-)}Q^{\ell'}_{j,k}(ds,du)\nonumber\\
&=\frac{1}{N} \sum_{j=1}^{S^{N}_k(0)}\int_0^t\int_0^\infty\lambda_{j,k}(t-s)\sum_{\ell'} [{\bf1}_{A^N_{j,k}(s^-)=0}{\bf1}_{u\le\bar{\Gamma}^{N,\ell'}_k(s^-)}\nonumber\\ &\qquad\qquad\qquad
-{\bf1}_{A_{j,k}(s^-)=0}{\bf1}_{u\le\bar{\Gamma}^{\ell'}_k(s^-)}]{\bf1}_{X_{j,k}(s)=\ell'}{\bf1}_{Y^{s,\ell'}_{j,k}(t) =\ell}Q^{\ell'}_{j,k}(ds,du)\\
&\quad+\frac{1}{N} \sum_{j=1}^{S^{N}_k(0)}\int_0^t\int_0^\infty\lambda_{j,k}(t-s)
{\bf1}_{A_{j,k}(s^-)=0}
\sum_{\ell'} {\bf1}_{u\le\bar{\Gamma}^{\ell'}_k(s^-)}{\bf1}_{X_{j,k}(s)=\ell'}{\bf1}_{Y^{s,\ell'}_{j,k}(t) =\ell}Q^{\ell'}_{j,k}(ds,du)\,.\nonumber
\end{align}
In order to bound the first term on the right of \eqref{decomp-FN}, we first deduce from \eqref{est-barF} and \eqref{eqn-barB-lowerbound} that there exists a constant $C_T$ such that for all $0\le t\le T$, $k\in\sK$ and $\ell\in\sL$, $\bar{\Gamma}^\ell_k(t)\le C_T$. It then follows that, denoting by $\kappa^{N,\ell}_k(t)$ the first term on the right of \eqref{decomp-FN},
\begin{align}\label{estimsupFN}
\E\left[\sup_{0\le s\le t\wedge \sigma^N_T}|\kappa^{N,\ell}_k(s)|\right]
& \le\lambda^\ast\frac{C_T}{N}\E\left[\sum_{j=1}^{S^{N}_k(0)}\sup_{0\le s\le t\wedge\tT}|A_{j,k}(s)-A^N_{j,k}(s)| \right] \non \\
& \quad +\lambda^\ast\sum_{\ell'}\E \left[\int_0^{t\wedge\tT}|\bar{\Gamma}^{N,\ell'}_k(s)-\bar{\Gamma}^{\ell'}_k(s)|ds\right]\,,
\end{align}
while, thanks to the law of large numbers, the second term on the right of \eqref{decomp-FN} converges a.s., as $N\to\infty$, towards 
\begin{align*}
&  \bar{S}_k(0) \E\left[ \int_0^t  \lambda_{1,k}(t-s) {\bf1}_{A_{1,k}(s^-)=0} \sum_{\ell'}  {\bf1}_{X_{1,k}(s)=\ell'}{\bf1}_{Y^{s,\ell'}_{1,k}(t) =\ell}  \bar{\Gamma}^{\ell'}_k(s)ds   \right]  \\
&=   \bar{S}_k(0) \int_0^t\bar{\lambda}_k(t-s)  \E\left[   {\bf1}_{A_{1,k}(s^-)=0} \sum_{\ell'}  {\bf1}_{X_{1,k}(s)=\ell'}{\bf1}_{Y^{s,\ell'}_{1,k}(t) =\ell}    \right]  \bar{\Gamma}^{\ell'}_k(s)ds \\ 
&= \int_0^t\bar{\lambda}_k(t-s)\sum_{\ell'}\bar{S}^{\ell'}_k(s)\bar{\Gamma}^{\ell'}_k(s)q^{\ell',\ell}_k(s,t)ds\,, 
\end{align*}
where we have used the fact that $\P\left(A_{k}(s)=0|X_k\right)=\exp\left(-\int_0^s\bar{\Gamma}^{X_k(r)}_k(r)dr\right)$ and formula \eqref{eq:FKr}.

Combining the last estimates with \eqref{conv:FN} yields
 \begin{equation} \label{eqn-barF-diff-bound-1}
 \begin{split}
   \E\left[{\bf1}_{t<\tT}\Big| \bar{\mathfrak{F}}^{N,\ell}_{k}(t)  - \bar{\mathfrak{F}}^{\ell}_{k}(t) \Big|\right] & \le 
     \ep_N+C\,\sum_{\ell'}\E \left[\int_0^{t\wedge\tT}|\bar{\Gamma}^{N,\ell'}_k(s)-\bar{\Gamma}^{\ell'}_k(s)|ds\right]\\&\quad +
 C \sum_{\ell'}  \frac{1}{N}\E\left[\sum_{j=1}^{S^N_k(0)}\sup_{0 \le r \le t\wedge\tT}  \Big|A^{N,\ell'}_{j,k}(r)  - A^{\ell'}_{j,k}(r) \Big| \right] \,. 
  \end{split}
 \end{equation}
 
 \begin{remark} \label{rem-proof-gamma0}
 In the case $\gamma=0$, instead of \eqref{eqn-Gamma-diff-bound}, we have the simpler bound
\begin{align*} 
\big|\bar\Gamma^{N,\ell}_{k}(t) - \bar\Gamma^{\ell}_{k}(t)\big|  
& \le  \beta^*\sum_{\ell',k'}   \Big| \bar{\mathfrak{F}}^{N,\ell'}_{k'}(t)  - \bar{\mathfrak{F}}^{\ell'}_{k'}(t) \Big| \,.  
\end{align*}
Hence combining this estimate and \eqref{eqn-barF-diff-bound-1}, 
by Gronwall's Lemma, we obtain 
\begin{align*}
\sum_{\ell,k} \E\left[{\bf1}_{t<\tT}\big|\bar{\Gamma}^{N,\ell}_{k}(t) - \bar{\Gamma}^{\ell}_{k}(t)\big|\right]\le \ep_N+C
 \sum_{\ell,k}  \frac{1}{N}\E\left[\sum_{j=1}^{S^N_k(0)}\sup_{0 \le r \le t\wedge\tT}  \Big|A^{N,\ell}_{j,k}(r)  - A^{\ell}_{j,k}(r) \Big| \right]\,.
\end{align*}
In that case, we do not need the estimate on $\bar{B}^{N,\ell}_k$ in Lemma \ref{le:BN>c}, nor the stopping time $\tT$, nor the estimates in the next three Lemmas to complete the proof.
\end{remark}

\bigskip

 It remains to consider the second term on the right of \eqref{eqn-Gamma-diff-bound}. 
 Observe that 
\begin{align} \label{eqn-barB-diff}
\frac{1}{(\bar{B}^{N,\ell'}_k(t))^\gamma}  - \frac{1}{(\bar{B}^{\ell'}_k(t))^\gamma} 
  =-\gamma \bigg(\int_0^1 \Big(u\bar{B}^{N,\ell'}_k(t) + (1-u) \bar{B}^{\ell'}_k(t)\Big)^{-\gamma-1} du \bigg)   \Big( \bar{B}^{N,\ell'}_k(t) - \bar{B}^{\ell'}_k(t)\Big). 
\end{align}
It is clear that on the event $\{t<\tT\}$, the integral on the right hand side is bounded by $(C^\ast_T)^{-\gamma-1}$.

Hence it follows from \eqref{eqn-Gamma-diff-bound}, \eqref{eqn-barF-diff-bound-1}, \eqref{eqn-barB-diff}  that, for all $ t \in [0,T] $,
\begin{align}\label{estimGamma}
\sum_{\ell,k} \E\left[{\bf1}_{t<\tT}\big|\bar\Gamma^{N,\ell}_{k}(t) - \bar\Gamma^{\ell}_{k}(t)\big|\right]&\le \ep_N+C
 \sum_{\ell,k}  \frac{1}{N}\E\left[\sum_{j=1}^{S^N_k(0)}\sup_{0 \le r \le t\wedge\tT}  \Big|A^{N,\ell}_{j,k}(r)  - A^{\ell}_{j,k}(r) \Big| \right]\nonumber \\
 &\qquad+C\sum_{\ell,k}\E\left[\int_0^{t\wedge\tT}|\bar{\Gamma}^{N,\ell}_k(s)-\bar{\Gamma}^{\ell}_k(s)|ds\right]  \\
 & \qquad 
 +C \sum_{\ell,k}\E\left[{\bf1}_{t<\tT}\left|\bar{B}^{N,\ell}_k(t) - \bar{B}^{\ell}_k(t)\right|\right]   \,.\nonumber
\end{align}
It will follow from the next three Lemmas that \eqref{estimGamma} holds without the last term. Hence from Gronwall's Lemma
\begin{align*}
\sum_{\ell,k} \E\left[{\bf1}_{t<\tT}\big|\bar\Gamma^{N,\ell}_{k}(t) - \bar\Gamma^{\ell}_{k}(t)\big|\right]\le \ep_N+C
 \sum_{\ell,k}  \frac{1}{N}\E\left[\sum_{j=1}^{S^N_k(0)}\sup_{0 \le r \le t\wedge\tT}  \Big|A^{N,\ell}_{j,k}(r)  - A^{\ell}_{j,k}(r) \Big| \right].
\end{align*}
The result follows by combining
 this last estimate with \eqref{eqn-A-jk-sup-bound} and Gronwall's Lemma.
\end{proof}

In the next lemmas, the same sequence $\ep_N$ and the constant $C$ may vary from one line to another.

\begin{lemma}\label{convS}
For any $T>0$, there exists $C>0$ and a sequence $\ep_N$ of positive numbers which tends to $0$ as $N\to\infty$, and such that for any $0\le t\le T$ and $k\in\sK$,
\begin{align*}
\sum_\ell\E\left[{\bf1}_{t<\tT}\left|\bar{S}^{N,\ell}_k(t)-\bar{S}^{\ell}_k(t)\right|\right]\le\ep_N+
\frac{C}{N}\sum_\ell \E\left[ \sup_{s\le t\wedge\tT}\sum_{j=1}^{S^N_k(0)}  \Big|A^{N,\ell}_{j,k}(s) - A^{\ell}_{j,k}(s) \Big|\right]\,.
\end{align*}
\end{lemma}
\begin{proof}
It follows from Gronwall's Lemma that it suffices to show 
\begin{align*}
\E\left[{\bf1}_{t<\tT}\left|\bar{S}^{N,\ell}_k(t)-\bar{S}^{\ell}_k(t)\right|\right]\le
&\frac{1}{N} \E\left[ {\bf1}_{t<\tT}\sum_{j=1}^{S^N_k(0)}  \Big|A^{N,\ell}_{j,k}(t) - A^{\ell}_{j,k}(t) \Big|\right]\\
&\quad+C\sum_{\ell}\E\left[\int_0^{t\wedge\tT}\left|\bar{S}^{N,\ell}_k(s)-\bar{S}^{\ell}_k(s)\right|ds\right] +\ep_N\,.
\end{align*}
We first note that
\begin{align*}
\bar{S}^{N,\ell}_k(t)&=\bar{S}^{N,\ell}_k(0)-\frac{1}{N}\sum_{j=1}^{S^{N,\ell}_k(0)}A^{N,\ell}_{j,k}(t)
-\sum_{\ell'\neq \ell}\frac{1}{N}P^{\ell,\ell'}_{S,k}\left(N\int_0^t\nu^{\ell,\ell'}_{S,k}(s)\bar{S}^{N,\ell}_k(s)ds\right)\\
&\quad+\sum_{\ell'\neq \ell}\frac{1}{N}P^{\ell',\ell}_{S,k}\left(N\int_0^t\nu^{\ell',\ell}_{S,k}(s)\bar{S}^{N,\ell'}_k(s)ds\right)\,,
\\
\bar{S}^{\ell}_k(t)&=\bar{S}^{\ell}_k(0)-\int_0^t\bar{S}^{\ell}_k(s)\bar{\Gamma}^\ell_k(s)ds
+\sum_{\ell'=1}^L\int_0^t\nu^{\ell',\ell}_{S,k}(s)\bar{S}^{\ell'}_k(s)ds\,. 
\end{align*}
It is clear that as $N\to\infty$, the following convergences hold in $L^1(\Omega)$:
\begin{align*}
\bar{S}^{N,\ell}_k(0)&\to\bar{S}^{\ell}_k(0)\,,\\
\sum_{\ell'=1, \ell'\neq \ell}^L\frac{1}{N}P^{\ell,\ell'}_{S,k}\left(N\int_0^t\nu^{\ell,\ell'}_{S,k}(s)\bar{S}^{\ell}_k(s)ds\right)
&\to \sum_{\ell'=1, \ell'\neq \ell}^L\int_0^t\nu^{\ell,\ell'}_{S,k}(s)\bar{S}^{\ell}_k(s)ds\,,\\
 \sum_{\ell'=1, \ell'\neq \ell}^L\frac{1}{N}P^{\ell',\ell}_{S,k}\left(N\int_0^t\nu^{\ell',\ell}_{S,k}(s)\bar{S}^{\ell'}_k(s)ds\right)
 &\to\sum_{\ell'=1, \ell'\neq \ell}^L\int_0^t\nu^{\ell',\ell}_{S,k}(s)\bar{S}^{\ell'}_k(s)ds\,,
\end{align*}
and the following identity holds:
\[ -\sum_{\ell'\neq\ell}\int_0^t\nu^{\ell,\ell'}_{S,k}(s)\bar{S}^{\ell}_k(s)ds+\sum_{\ell'\neq \ell} \int_0^t\nu^{\ell',\ell}_{S,k}(s)\bar{S}^{\ell'}_k(s)ds=\sum_{\ell'=1}^L\int_0^t\nu^{\ell',\ell}_{S,k}(s)\bar{S}^{\ell'}_k(s)ds\,.\]
Moreover, for all $ \ell'\neq \ell $,
\begin{align*}
\E\Bigg[&{\bf1}_{t<\tT} \left|\frac{1}{N}P^{\ell,\ell'}_{S,k}\left(N\int_0^t\nu^{\ell,\ell'}_{S,k}(s)\bar{S}^{N,\ell}_k(s)ds\right)
-\frac{1}{N}P^{\ell,\ell'}_{S,k}\left(N\int_0^t\nu^{\ell,\ell'}_{S,k}(s)\bar{S}^{\ell}_k(s)ds\right)\right|\\
&+{\bf1}_{t<\tT} \left|\frac{1}{N}P^{\ell',\ell}_{S,k}\left(N\int_0^t\nu^{\ell',\ell}_{S,k}(s)\bar{S}^{N,\ell'}_k(s)ds\right)
-\frac{1}{N}P^{\ell',\ell}_{S,k}\left(N\int_0^t\nu^{\ell',\ell}_{S,k}(s)\bar{S}^{\ell'}_k(s)ds\right)\right|\Bigg]\\
&\le C\E\left[\int_0^{t\wedge \tT} \left|\bar{S}^{N,\ell}_k(s)-\bar{S}^{\ell}_k(s)\right|ds+
\int_0^{t\wedge \tT} \left|\bar{S}^{N,\ell'}_k(s)-\bar{S}^{\ell'}_k(s)\right|ds\right]\,.
\end{align*}
Now consider the difference 
\[ \Delta^{N,\ell}_k(t):= -\frac{1}{N}\sum_{j=1}^{S^N_k(0)}A^{N,\ell}_{j,k}(t)+\int_0^t\bar{S}^{\ell}_k(s)\bar{\Gamma}^\ell_k(s)ds\,.\]
We rewrite it as 
\begin{align*}
\Delta^{N,\ell}_k(t)&=\frac{1}{N}\sum_{j=1}^{S^N_k(0)}[A^{\ell}_{j,k}(t)-A^{N,\ell}_{j,k}(t)]+\int_0^t\bar{S}^{\ell}_k(s)\bar{\Gamma}^\ell_k(s)ds-\frac{1}{N}\sum_{j=1}^{S^N_k(0)}A^{\ell}_{j,k}(t)\,.
\end{align*}
It remains to show that, as $N\to\infty$,
\begin{equation*}\label{convAS}
\frac{1}{N}\sum_{j=1}^{S^N_k(0)}A^{\ell}_{j,k}(t)\to\int_0^t\bar{S}^{\ell}_k(s)\bar{\Gamma}^\ell_k(s)ds
\end{equation*} 
in $L^1(\Omega)$. We note that 
\begin{align*}
\frac{1}{N}\sum_{j=1}^{S^N_k(0)}A^{\ell}_{j,k}(t)&\to  \bar{S}_k(0) \E\left[A^\ell_k(t)\right]\\
&= \bar{S}_k(0) \int_0^t\E\left[{\bf1}_{X_k(s)=\ell}{\bf1}_{A_k(s)=0}\right]\bar{\Gamma}^\ell_k(s)ds\\
&=\bar{S}_k(0)\int_0^t\E\left[{\bf1}_{X_k(s)=\ell} \exp\left(-\int_0^s\bar{\Gamma}^{X_k(r)}_k(r)dr\right)\right]\bar{\Gamma}^\ell_k(s)ds\\
&=\int_0^t\bar{S}^{\ell}_k(s)\bar{\Gamma}^\ell_k(s)ds\,,
\end{align*}
where we have used successively the strong law of large numbers, the first line of \eqref{eq:McKV}, the fact that 
$\P(A_k(s)=0|X_k(\cdot))=\exp\left(-\int_0^s\bar{\Gamma}^{X_k(r)}_k(r)dr\right)$, and formula \eqref{eq:FKr} from Proposition \ref{pro:FK}.
\end{proof}

\begin{lemma}\label{convI}
For any $T>0$, there exists a sequence $\ep_N$ of positive numbers which tends to $0$ as $N\to\infty$, and such that for any $0\le t\le T$,  $k\in\sK$,
\begin{align*}
\sum_{\ell}\E\left[{\bf1}_{t<\tT}\left|\bar{I}^{N,\ell}_k(t)-\bar{I}^{\ell}_k(t)\right|\right]&\le\ep_N+
\frac{C}{N} \sum_\ell\E\left[ \sup_{s\le t\wedge\tT}\sum_{j=1}^{S^N_k(0)}  \Big|A^{N,\ell}_{j,k}(s) - A^{\ell}_{j,k}(s) \Big|\right]\\&\quad+C\sum_\ell\E\left[\int_0^{t\wedge\tT}|\bar{\Gamma}^{N,\ell}_k(r)-\bar{\Gamma}^{\ell}_k(r)|dr\right]\,.
\end{align*}
\end{lemma}
\begin{proof}
Again, it suffices to show that
\begin{align*}
\E\left[{\bf1}_{t<\tT}\left|\bar{I}^{N,\ell}_k(t)-\bar{I}^{\ell}_k(t)\right|\right]&\le\ep_N+
\frac{C}{N} \E\left[{\bf1}_{t<\tT} \sum_{j=1}^{S^N_k(0)}\sum_{\ell'=1}^L \Big|A^{N,\ell'}_{j,k}(t) - A^{\ell'}_{j,k}(t) \Big| \right]\\
&\qquad+C\sum_{\ell'}\E\left[ \int_0^{t\wedge\tT}\left|\bar{I}^{N,\ell'}_k(s)-\bar{I}^{\ell'}_k(s)\right|ds\right] \\
& \qquad
+C\sum_\ell\E\left[ \int_0^{t\wedge\tT}|\bar{\Gamma}^{N,\ell}_k(r)-\bar{\Gamma}^{\ell}_k(r)|dr\right]\,.
\end{align*}
Four of the terms in the equation for $\bar{I}^{N,\ell}_k(t)$ (see \eqref{eq:IN}) are treated exactly as in the previous Lemma. Moreover, 
by the strong law of large numbers,
\[\frac{1}{N}\sum_{i=1}^{I^{N,\ell'}_k(0)}{\bf1}_{\eta_{-i,k}\le t}{\bf1}_{Y^{0,\ell'}_{i,k}(\eta_{-i,k})=\ell}
\to \bar{I}^{\ell'}_k(0) \int_0^t q_k^{\ell',\ell}(0,s)F^0_k(ds)\] 
 a.s. in $\bD$. It remains to consider the term
\begin{align*}
\sI^{N,\ell}_{k}(t):=\frac{1}{N}\sum_{j=1}^{S^{N}_{k}(0)} \bone_{\tau^{N}_{j, k}+ \eta_{j,k} \le t} \sum_{\ell'}\bone_{X_{j,k}(\tau^N_{j,k}) =\ell'}
 \bone_{Y^{\tau^{N}_{j, k},\ell'}_{j,k}\!\!(\tau^{N}_{j, k}+\eta_{j,k})=\ell} \,\,.
\end{align*}
For that sake, we introduce a new collection of i.i.d. PRMs $\tilde{Q}^\ell_{j,k}$ on $\R^3_+$, for
$k\in\sK$, $\ell\in\sL$ and $j\ge1$  with mean measure $dsduF_k(dr)$. The PRM  $\tilde{Q}^\ell_{j,k}$ 
is defined as follows. Let $\{s^{\ell,j,k}_i,u^{\ell,j,k}_i,\ i\ge1\}$ be some enumeration of the points of the 
PRM $Q^\ell_{j,k}$. Let $\{\eta_i,\ i\ge1\}$ be some sequence of i.i.d. r.v.'s which are globally independent 
of the PRM $Q^\ell_{j,k}$, and all have the distribution $F_k(dr)$. The PRM $\tilde{Q}$ on $\R_+^3$ is given as
\[ \tilde{Q}=\sum_{i=1}^\infty \delta_{s^{\ell,j,k}_i,u^{\ell,j,k}_i,\eta_i}\,.\] 
With this new PRM, we have 
\[ \sI^{N,\ell}_{k}(t)=\frac{1}{N}\sum_{j=1}^{S^{N}_{k}(0)}\sum_{\ell'}\int_0^t\int_0^\infty\int_0^{t-s}
{\bf1}_{A^N_{j,k}(s^-)=0}{\bf1}_{X_{j,k}(s)=\ell'}{\bf1}_{Y^{s,\ell'}_{j,k}(s+r)=\ell}{\bf1}_{u\le\bar{\Gamma}^{N,\ell'}_k(s^-)}\tilde{Q}^{\ell'}_{j,k}(ds,du,dr)\,.\]
Note that for each $j$ and $\ell'$, the integral is either $0$ or $1$, which will allow us to simplify the difference with a similar integral. We have
\begin{equation}\label{sIN}
\begin{split}
\sI^{N,\ell}_{k}(t) &=\frac{1}{N}\sum_{j=1}^{S^{N}_{k}(0)}\sum_{\ell'}\int_0^t\int_0^\infty\int_0^{t-s}
{\bf1}_{X_{j,k}(s)=\ell'}{\bf1}_{Y^{s,\ell'}_{j,k}(s+r)=\ell}\Big[{\bf1}_{A^N_{j,k}(s^-)=0}{\bf1}_{u\le\bar{\Gamma}^{N,\ell'}_k(s^-)}\\
&\qquad\qquad\qquad\qquad \qquad \qquad \qquad  -{\bf1}_{A_{j,k}(s^-)=0}{\bf1}_{u\le\bar{\Gamma}^{\ell'}_k(s^-)}\Big]\tilde{Q}^{\ell'}_{j,k}(ds,du,dr)\\
&\quad+\frac{1}{N}\sum_{j=1}^{S^{N}_{k}(0)}\sum_{\ell'}\int_0^t\int_0^\infty\int_0^{t-s}
{\bf1}_{A_{j,k}(s^-)=0}{\bf1}_{X_{j,k}(s)=\ell'}{\bf1}_{Y^{s,\ell'}_{j,k}(s+r)=\ell}{\bf1}_{u\le\bar{\Gamma}^{\ell'}_k(s^-)}\tilde{Q}^{\ell'}_{j,k}(ds,du,dr)\,.
\end{split}
\end{equation}
The expectation of the absolute value of the first term on the right of \eqref{sIN} evaluated at time $t\wedge\tT$ is bounded by 
\[\frac{1}{N}\E\left[ \sum_{j=1}^{S^{N}_{k}(0)}\sum_{\ell'}\sup_{0\le s\le t\wedge\tT}\Big|A^{N,\ell'}_{j,k}(s)-A^{\ell'}_{j,k}(s)\Big| \right]
+\sum_{\ell'}\E\left[ \int_0^{t\wedge\tT}\big|\bar{\Gamma}^{\ell'}_k(s)-\bar{\Gamma}^{N,\ell'}_k(s)\big|ds\right].\]

From the law of large numbers, the second term in \eqref{sIN} converges as $N\to\infty$, towards
\begin{align*}
 \bar{S}_k(0) \sum_{\ell'}\E&\left[\int_0^t\int_0^{t-s}{\bf1}_{A_{k}(s)=0}{\bf1}_{X_{k}(s)=\ell'}{\bf1}_{Y^{s,\ell'}_{k}(s+r)=\ell}\bar{\Gamma}^{\ell'}_k(s)dsF_k(dr)\right]\\
&=\sum_{\ell'}\int_0^t\int_0^{t-s}\bar{S}^{\ell'}_k(s)\bar{\Gamma}^{\ell'}_k(s)q^{\ell',\ell}_k(s,s+u)F_k(du)ds,
\end{align*}
where we have used the fact that $\P\left(A_{k}(s)=0|X_k\right)=\exp\left(-\int_0^s\bar{\Gamma}^{X_k(r)}_k(r)dr\right)$ and formula \eqref{eq:FKr}.
The result follows.
\end{proof}

\begin{lemma}\label{convR}
For any $T>0$, there exists a sequence $\ep_N$ of positive numbers which tends to $0$ as $N\to\infty$, and such that for  any $0\le t\le T$, and $k\in \sK$, 
\begin{align*}
\sum_{\ell}\E\left[{\bf1}_{t<\tT}\left|\bar{R}^{N,\ell}_k(t)-\bar{R}^{\ell}_k(t)\right|\right]&\le\ep_N+
\frac{C}{N} \sum_\ell\E\left[ \sup_{s\le t\wedge\tT}\sum_{j=1}^{S^N_k(0)}  \Big|A^{N,\ell}_{j,k}(s) - A^{\ell}_{j,k}(s) \Big|\right]\\&\qquad  +C\sum_\ell\E\left[ \int_0^{t\wedge\tT}|\bar{\Gamma}^{N,\ell}_k(r)-\bar{\Gamma}^{\ell}_k(r)|dr\right]\,.
\end{align*}
\end{lemma}
The proof of this Lemma is very similar to that of the previous one, as the reader can easily verify.

\subsection{Completing the proof of Theorem \ref{thm-LLN}} \label{sec-completing}

It follows from the above Lemmas that for each $t>0$, $(\bar{S}^{N,\ell}_k(t), \bar{\mathfrak{F}}^{\ell}_k(t), \bar{I}^{N,\ell}_k(t),\bar{R}^{N,\ell}_k(t))$ converges in probability as $N\to\infty$ towards $(\bar{S}^{\ell}_k(t), \bar{\mathfrak{F}}^{N,\ell}_k(t), \bar{I}^{\ell}_k(t),\bar{R}^{\ell}_k(t))$. It remains to prove that the convergences hold in $\bD$.

We first consider $\bar{\mathfrak{F}}^{N,\ell}_k$. Consider the right hand side of \eqref{conv:FN}. The first term tends to $0$ a.s., locally uniformly in $t$, thanks to Theorem 1 in \cite{rao1963law}.  
Concerning the second term on the right of \eqref{conv:FN}, its convergence follows from the decomposition \eqref{decomp-FN}. The first term on the right 
of \eqref{decomp-FN} converges to $0$ in probability locally uniformly in $t$, thanks to \eqref{estimsupFN} and the rest of the proof of Lemma \ref{lem-An-Ai-diffbound}, while the locally uniform convergence in $t$ in probability of the second term follows again from Theorem 1 in \cite{rao1963law}.

We now establish the convergence in $\bD$ of the other quantities. 
We shall next use repeatedly Dini's theorem, which implies that a sequence of increasing functions which converges pointwise to a continuous function, converges in fact locally uniformly. This applies to random functions which converge in probability, since convergence in probability is equivalent to the fact that from any subsequence, one can extract a further subsequence which converges a.s. Also note that at most one limit term in each equation is discontinuous, so we will have no difficulty in adding convergences in $\bD$.

We next consider the process  $\bar{S}^{N,\ell}_k(t)$. 
Let us first discuss the Poisson terms, which are of the form
\[ \frac{1}{N}P^{\ell,\ell'}_{S,k}\left(\int_0^t\nu^{\ell,\ell'}_{S,k}(s) {S}^{N,\ell}_k(s)ds\right)
=\frac{1}{N}P^{\ell,\ell'}_{S,k}\left(N\int_0^t\nu^{\ell,\ell'}_{S,k}(s) \bar{S}^{N,\ell}_k(s)ds\right)\,,\]
which are non-decreasing and from the LLN for Poisson processes converge in probability,  towards the continuous function
$\int_0^t\nu^{\ell,\ell'}_{S,k}(s) \bar{S}^{\ell}_k(s)ds$. Hence, from the second Dini theorem, the convergence is locally uniform in $t$.

We finally need to consider the term from \eqref{eq:SN}:
\begin{align*}
\sum_{j=1}^{S^N_k(0)}A^{N,\ell}_{j,k}(t)&=\sum_{j=1}^{S^N_k(0)}\int_0^t\int_0^\infty{\bf1}_{A^{N}_{j,k}(s^-)=0}{\bf1}_{X_{j,k}(s)=\ell}{\bf1}_{u\le \bar\Gamma^{N,\ell}_{k}(s^-)} Q^\ell_{j,k}(ds,du)\,. 
\end{align*}
There exists a standard PRM $Q^\ell_{k}(ds,du)$ on $\R_+^2$ such that
\begin{align*}
\frac{1}{N}\sum_{j=1}^{S^N_k(0)}A^{N,\ell}_{j,k}(t)&=\frac{1}{N}\sum_{j=1}^{S^N_k(0)}\int_0^t\int_0^\infty{\bf1}_{u\le {\bf1}_{A^{N}_{j,k}(s^-)=0}{\bf1}_{X_{j,k}(s)=\ell}\bar\Gamma^{N,\ell}_{k}(s^-)} Q^\ell_{j,k}(ds,du)\\
&=\frac{1}{N}\int_0^t\int_0^\infty{\bf1}_{u\le {S}^{N,\ell}_k(s^-)\bar\Gamma^{N,\ell}_{k}(s^-)} Q^\ell_{k}(ds,du)\\
&=\int_0^t\bar{S}^{N,\ell}_k(s)\bar\Gamma^{N,\ell}_{k}(s)ds
+\frac{1}{N}\int_0^t\int_0^\infty{\bf1}_{u\le {S}^{N,\ell}_k(s^-)\bar\Gamma^{N,\ell}_{k}(s^-)} \bar{Q}^\ell_{k}(ds,du),
\end{align*}
where we have used the fact that $S^{N,\ell}_k(s)=\sum_{j=1}^{S^N_k(0)}{\bf1}_{A^{N}_{j,k}(s)=0}{\bf1}_{X_{j,k}(s)=\ell}$, and $\bar{Q}^\ell_{k}(ds,du)=Q^\ell_{k}(ds,du)-dsdu$.
It follows from Lemmas \ref{convS}, \ref{lem-An-Ai-diffbound} and \ref{le:BN>c},   that the first term on the right converges locally uniformly in $t$ in probability towards $\int_0^t\bar{S}^{\ell}_k(s)\bar\Gamma^{\ell}_{k}(s)ds$, while the second term is a martingale which converges locally uniformly in $t$ towards $0$ in probability.

We next consider the process $\bar{I}^{N,\ell}_k(t)$. There are two new terms in \eqref{eq:IN}, compared to \eqref{eq:SN}. The first one is
\[\sum_{\ell'=1}^L \sum_{i=1}^{I^{N,\ell'}_k(0)} \bone_{\eta^{0}_{i,k} \le t} \bone_{Y^{0,\ell'}_{i,k}(\eta^{0}_{i,k}) =\ell} \,\,,\]
and the second one
\[ \sum_{j=1}^{S^{N}_{k}(0)} \bone_{\tau^{N}_{j, k}+ \eta_{j,k} \le t} \sum_{\ell'}\bone_{X_{j,k}(\tau^N_{j,k}) =\ell'}
 \bone_{Y^{\tau^{N}_{j, k},\ell'}_{j,k}\!\!(\tau^{N}_{j, k}+\eta_{j,k})=\ell}\,.  \]
 It follows from the law of large numbers in $\bD$ (see, e.g., \cite{rao1963law}) that the first term
 converges in probability in $\bD$ towards
 \[\sum_{\ell'=1}^L \bar{I}^{\ell'}_k(0) \int_0^t  q_k^{\ell', \ell}(0,s)  F^0_{k}(ds)\,\,.\]
It remains to reconsider the argument used to treat $\sI^{N,\ell}_{j,k}(t)$ in the proof of Lemma \ref{convI}. 
The uniformity in $t$ of the convergence in probability to $0$ of the first term on the right hand side of 
\eqref{sIN} is rather obvious. Concerning the second term, the uniformity in $t$ of the convergence will follow from Dini's theorem, if we show that the mapping 
\[t\mapsto G(t):=\int_0^t\int_0^{t-s}\bar{S}^{\ell'}_k(s)\bar{\Gamma}^{\ell'}_k(s)q^{\ell',\ell}_k(s,s+u)F_k(du)ds\]
is continuous. But for $t'<t$, if $g(s,u):=\bar{S}^{\ell'}_k(s)\bar{\Gamma}^{\ell'}_k(s)q^{\ell',\ell}_k(s,s+u)$,
we have with $C:=KL\lambda^\ast\beta^\ast$,
\begin{align*}
G(t)-G(t')&=\int_{\R_+^2}{\bf1}_{t'<s+u\le t}g(s,u)F_k(du)ds\\
&\le C \int_{\R_+^2}{\bf1}_{t'<s+u\le t}F_k(du)ds\\
&\le C (t-t')\,,
\end{align*}
where the last inequality follows by integrating first with respect to $ s $.

\section{Appendix}
The aim of this section is to establish the following Lemma.
\begin{lemma}\label{le:continuity}
If $h\in\bD$, $g$ is measurable from $\R_+$ into itself and for some $C>0$, $0\le h(t)\le C$, $0\le g(t)\le C$ for all $t\ge0$, then 
\[ t\mapsto \int_0^t h(t-s)g(s)ds\]
is continuous.
\end{lemma}
\begin{proof}
Let $t_n\to t$ as $n\to\infty$. 
\begin{align*}
& \left|\int_0^t h(t-s)g(s)ds-\int_0^{t_n} h(t_n-s)g(s)ds\right| \\
&\le\int_{t\wedge t_n}^{t\vee t_n}h(t\vee t_n-s)g(s)ds
+\int_0^t|h(t-s)-h(t_n-s)|g(s){\bf1}_{s\le t_n}ds\,.
\end{align*}
The first term on the right is bounded by $C^2|t-t_n|$, while the integrand in the second term is bounded 
and converges to $0$ $ds$ a.e.
\end{proof}

This Lemma is used above in conjunction with the following remark. If $(f_n,g_n)\to(f,g)$ in $\bD^2$ and $g$ is continuous, then $f_n+g_n\to f+g$ in $\bD$. Indeed, for any $T>0$, let $\lambda_{n,T}$ be the time change from $[0,T]$ into itself, which is such that $\sup_{0\le t\le T}|\lambda_{n,T}(t)-t|\to0$ and 
$\sup_{0\le t\le T}|f_n\circ\lambda_{n,T}(t)-f(t)|\to0$. Since $g$ is continuous, $\sup_{0\le t\le T}|g_n(t)-g(t)|\to0$, and also $\sup_{0\le t\le T}|g_n\circ\lambda_{n,T}(t)-g(t)|\to0$, and we conclude that
$\sup_{0\le t\le T}|(f_n+g_n)\circ\lambda_{n,T}(t)-f(t)+g(t)|\to0$, which implies the result. Note that if both $f$ and $g$ are discontinuous and have a common jump, the two convergences may involve two incompatible time changes, and the sum may not converge in $\bD$. 

\bibliographystyle{plain}

\bibliography{Epidemic-Age,biblio_raphael}

\end{document}